\documentclass[reqno]{amsart}
\usepackage{amssymb,latexsym}
\usepackage{amsmath}
\usepackage{amsthm}
\usepackage{graphicx}
\usepackage{hyperref}
\usepackage{titletoc}
\usepackage{cancel}
\usepackage{color,xcolor}
\usepackage{epsfig}
\usepackage{epstopdf}
\usepackage{bm}
\usepackage[pagewise]{lineno}
\usepackage{CJK}
\usepackage{subfigure}
\usepackage[graphicx]{realboxes}
\usepackage{tikz}

\numberwithin{equation}{section}
\newtheorem{theorem}{Theorem}[section]

\newtheorem{lemma}[theorem]{Lemma}
\newtheorem{corollary}[theorem]{Corollary}

\def\begeq{\begin{equation}}
\def\endeq{\end{equation}}

\allowdisplaybreaks

\title[Gradient estimates for $\Delta_pu+\Delta_qu+h(u,|\nabla u|^2)=0$]{Liouville type theorems for some $(p,q)$-Laplace equations with gradient dependent reaction on Riemannian manifolds}
\author{Youde Wang$^\S$}
\thanks{$\S\,$ Corresponding author: 1. School of Mathematics and Information Sciences, Guangzhou University;
		2. State Key Laboratory of Mathematical Sciences (SKLMS), Academy of Mathematics and Systems Science, and School of
		Mathematical Sciences, UCAS, Beijing 100190, China.
Email: wyd@math.ac.cn}

\author{Liqin Zhang$^\dag$}
\thanks{$\ddag\,$ School of Mathematics and Information Science,
Guangzhou University, Guangzhou 510006, P. R. China.
Email: 1061837643@qq.com}

\begin{document}
\maketitle
\rm\begin{abstract}\vspace{-2.5em}
{In this paper, we combine Bochner formula, Saloff-Coste's Sobolev inequality and the Nash-Moser iteration method to study the local and global behaviors of solutions to the nonlinear elliptic equation $\Delta_pu+\Delta_qu+h(u,|\nabla u|^2)=0$ defined on a complete Riemannian manifold $\left(M,g\right)$, where $q\ge p>1$, $h\in C^1(\mathbb{R}\times\mathbb{R}^{+})$ and $\Delta_z u=\mathrm{div}\left(\left|\nabla u\right|^{z-2}\nabla u\right)$, with $z\in\{ p,~q\}$, is the usual $z$-Laplace operator. Under some assumptions on $p$, $q$ and $h(x,y)$, we derive concise gradient estimates for solutions to the above equation and then obtain some Liouville type theorems. In particular, we use integral estimate method to show that, if $u$ is a non-negative entire solution to $\Delta_p u +\Delta_q u=0$ ($n\le p\le q$) on a complete non-compact Riemannian manifold $M$ with non-negative Ricci curvature and $\dim M = n\ge2$, then $u$ is a trivial constant solution.}
\end{abstract}

\medskip

MSC[2020] Primary 58J05, 35B53, 35J92; Secondary 35B45, 35B08, 35B09.

\textbf{\tableofcontents}

\section{\textbf{Introduction}}
\noindent Gradient estimate is a fundamental technique in the study of partial differential equations on a Riemannian manifold. Indeed, one can use gradient estimate to deduce Liouville type theorems (\cite{NS25, Yau, 4}), to derive Harnack inequalities (\cite{SW25, Yau, 4}), to study the geometry of manifolds (\cite{7, 4}), etc.

In this paper, we are concerned with the following $(p, q)$-equation defined on a complete Riemannian manifold $\left(M,g\right)$ equipped with a metric $g$
\begin{align}\label{1}
\Delta_pu+\Delta_qu+h(u,|\nabla u|^2)=0,
\end{align}
where $q\ge p>1$, $h\in C^1(\mathbb{R}\times\mathbb{R}^{+})$ and $\Delta_zu=\mathrm{div}\left(\left|\nabla u\right|^{z-2}\nabla u\right)$, with $z\in\{ p,~q\}$, is the usual $z$-Laplace operator. This equation is a degenerate quasilinear elliptic equation and includes some well-known equations, for instance, Hamilton-Jacobi equation (see \cite{Lions}).

\subsection{History and motivations}\

Since the content of this paper is closely concerned with double phase problems, we start with a short description on the background. The double-phase problem \eqref{1} is motivated by numerous models arising in mathematical physics. For instance, we can refer to Born-Infeld equation \cite{10} that appears in electromagnetism:
\begin{center}
$
-\mathrm{div} \left(\dfrac{\nabla u}{\sqrt{1-2\left|\nabla u\right|^2}}\right)=h(u)\quad \mbox{in}~\Omega,
$
\end{center}
where $\Omega\subset{\mathbb R}^n$ is a bounded domain with $C^2$-boundary $\partial\Omega$. Indeed, by the Taylor formula, we have
\begin{center}
$
\dfrac{1}{\sqrt{1-2x}}=1+x+\dfrac{3}{2}x^2+\cdots+\dfrac{(2m-3)!!}{(m-1)!}x^{m-1}+\cdots.
$
\end{center}
Taking $x=\left|\nabla u\right|^2$ and adopting the first order approximation, we obtain problem \eqref{1} for $p=2$ and $q=4$. Furthermore, the $m$-th order approximation problem is driven by the multi-phase differential operator
\begin{center}
$
-\Delta u-\Delta_4u-\dfrac{3}{2}\Delta_6u-\cdots-\dfrac{(2m-3)!!}{(m-1)!}\Delta_{2m}u.
$
\end{center}

Besides, the operator $\Delta_p + \Delta_q$ is also directly related to the following functionals:
$$\mathcal{F}(u)=\int_\Omega(a(x)|\nabla u|^p + |\nabla u|^q)dx,\quad u\in W^{1,\max\{p,q\}}(\Omega),$$
if $a(x)\equiv 1$. Such functionals $\mathcal{F}(u)$ were first studied by Marcellini \cite{Marc89, Marc91} and Zhikov \cite{Zhi87, Zhi95}, in the context of problems of the calculus of variations and of nonlinear elasticity theory, including the Lavrentiev gap phenomenon. The correponding
operators are used to describe diffusion-type processes, where certain subdomains are distinguished from others. For instance, this structure is used to describe a composite material having on $\{x\in \Omega: a(x) =0\}$ an energy density with $q$-growth, but on $\{x\in \Omega: a(x)>0\}$ the energy density has $p$-growth.
\medskip

Now let us recall some previous results which are closely related with our paper. Bobkov and Tanaka \cite{Bob} they also dealt with the study of positive solutions to the $(p, q)$-Laplace equation
\begin{align}\label{new}
-\Delta_ pu - \Delta_qu = \alpha_1 |u|^{p-2}u + \alpha_2 |u|^{q-2}u\quad\mbox{in}\,\,\Omega\subset\mathbb{R}^n,
\end{align}
subject to the Dirichlet boundary condition $u=0$ on $\partial\Omega$. They gave a complete description of two-dimensional sets in the $(\alpha_1,
\alpha_2)$-plane that correspond to the existence and the non-existence of positive solutions to the above problem. Also, in \cite{Bob1} they studied recently the multiplicity of positive solutions for $(p,q)$-Laplace equations.

B. Alreshidi, D. D. Hai and R. Shivaji in \cite{12} studied the existence of a positive solution to the $(p,q)$-Laplacian equation
\begin{center}
$
\begin{cases}
-\Delta_pu-\Delta_qu=\lambda_0 h_0(u),&\mbox{in}~\Omega,\\[3mm]
u=0,&\mbox{on}~\partial\Omega,
\end{cases}
$
\end{center}
where $\Omega$ is a bounded domain in ${\mathbb R}^n$ with smooth boundary $\partial\Omega$, $h_0:(0,+\infty)\longrightarrow\mathbb R$ is continuous, $p$-sublinear at $\infty$ and is allowed to be singular at $0$, and $\lambda_0>0$ is a large parameter. S. El Manouni, K. Perera and P. Winkert \cite{ElPW} studied recently the first eigenvalue of the $(p, q)$-Laplacian and some related problems.

Meanwhile, for the case the definition domain of $(p, q)$-Laplace equation is an Euclidean space $\mathbb{R}^n$ or a Riemannian manifold, V. Ambrosio \cite{A} consider the nonlinear $(p, q)$-Schr\"odinger equation
\begin{center}
$
\begin{cases}
-\epsilon^p\Delta_pu -\epsilon^q\Delta_qu + V(x)(|u|^{p-2}u + |u|^{q-2}u)=f(u), \quad x\in \mathbb{R}^n\\
0 <u\in W^{1,q}(\mathbb{R}^n)\cap W^{1,p}(\mathbb{R}^n)
\end{cases}
$
\end{center}
where $\epsilon> 0$ is a small parameter, $n\geq 3$, $1 < p < q < n$, $V$ is a continuous potential with $V_0 = \inf_{x\in\Lambda} V <
\min_{x\in\partial\Lambda}V$ for some bounded open set $\Lambda\in\mathbb{R}^n$, and $f$ is a continuous subcritical Berestycki-Lions type nonlinearity with $f(t) = 0$ for $t\leq0$. Using variational arguments, the author obtains a solution $u_\epsilon$ for the $(p,q)$-Schr\"odinger equation and studies the concentration of $u_\epsilon$ as $\epsilon\to 0$. Moreover, for  $\epsilon>0$ small,  by employing certain uniform estimates for the solutions of the modified problems, the author proves that the solutions constructed for the modified problems give rise to solutions of the original problem.

Also, V. Ambrosio \cite{A1} dealt with the following class of $(p, q)$-Laplacian problems:
$$-\Delta_pu-\Delta_qu = g(u) \quad\mbox{in}\,\,\mathbb{R}^n$$
with $u \in W^{1,p}(\mathbb{R}^n)\cap W^{1,q}(\mathbb{R}^n)$, where $n \geq 2$, $1 < p < q \leq n$, and $g: \mathbb{R} \to \mathbb{R}$ is a Berestycki-Lions type nonlinearity. Using appropriate variational arguments, he obtained the existence of a ground state solution and proved the existence of infinitely many radially symmetric solutions.

Very recently, the authors of this paper in \cite{WZ} discussed the gradient estimates of solutions to the following equation of the same type as above \eqref{new}, i.e.
$$-\Delta_ pu - \Delta_qu = a|u|^{s-1}u + b|u|^{l-1}u$$
defined on a complete noncompact Riemannian manifold with Ricci curvature bounded from below, where $q>p>1$, $a$, $b$, $s$ and $l$ are constants, and then derive some Liouville theorem. In particular, it is proved in \cite{WZ} that, if $u$ is a non-negative $(p, q)$-harmonic function on $M$, i.e., a non-negative entire solution to $\Delta_pu +\Delta_qu =0$ ($1<p<q$) on a complete non-compact Riemannian manifold $M$ with non-negative Ricci curvature and $\dim M = n \geq 3$, and $f = |\nabla u|^2 \in L^\beta(M)$, where $\beta> \delta(n, p, q)>0$, then $u$ is a trivial constant solution.

Later on, Yang and the authors of this paper in \cite{WYZh} shew that, if an entire positive solution $u$ to $\Delta_p u +\Delta_q u=0$ ($1<p\le q$) on a complete non-compact Riemannian manifold $M$ with non-negative Ricci curvature and $\dim M = n\ge3$ satisfies
\begin{align*}
\lim_{M\ni x\to\infty}\frac{u(x)}{d(x_0,x)}=0
\end{align*}
for some $x_0\in M$, where $d(x_0, x)$ denotes the distance between $x_0$ and $x$, then $u$ is a constant.

Shen and Zhu in \cite{Shen-Zhu} employed the methods in \cite{HWW} to prove the Liouville theorem for positive $(p, q)$-harmonic functions, provided
$$\frac{4(p-1)(q-1)}{(p-q)^2} > (n-1).$$

Besides, let us also mention that several studies have been devoted recently to the investigation of related problems and a lot of papers have appeared dealing with problems involving $(p,q)$-Laplacian in both bounded and unbounded domains. For the references therein, see e.g. Ambrosio et al \cite{AR},  Baldelli et al \cite{BBF1, BBF2}, Baldelli and Filippucci \cite{BF}, Cherfils and Il'yasov \cite{CI}, C. He and G. Li \cite{HL}. For more results, we refer to \cite{Mara} where a short account of recent existence and multiplicity theorems on the Dirichlet problem for an elliptic equation with $(p, q)$-Laplacian in a bounded domain is performed.
\medskip

For anisotropic $(p, q)$-Laplacian associated with functional $\mathcal{F}(u)$ mentioned above, Liu and N.S. Papageorgiou \cite{LP} considered a parametric nonlinear Dirichlet problem driven by the double phase differential operator and a reaction that has the competing effects of parametric ``concave" term and of a ``convex" perturbation (concave-convex problem), i.e.
$$-\Delta_p^au-\Delta_qu= \lambda|u|^{\tau-2}u+f(x,u)\quad\mbox{in}\,\,\Omega$$
with Dirichlet boundary value $u=0$ on $\partial\Omega$. Given $a\in C(\Omega\setminus\{0\})$ with $a(x) \geq 0$ for all $x\in\bar{\Omega}$, by $\Delta^a_p$ we denote the weighted $p$-Laplacian with weight $a(x)$, which is defined by
 $$\Delta_p^au = \mathrm{div}(a(x)|\nabla u|^{p-2}\nabla u).$$
Using variational tools together with truncation and comparison techniques and critical groups, we show that for all small values of the parameter, the problem has at least three nontrivial bounded solutions. More multiplicity results for double phase boundary value problems can be found in the works \cite{LP22, PPR22} and references therein. A detailed account of the progress made so far, can be found in the papers of Mingione-Radulescu \cite{MR21} and Ragusa-Tachikawa \cite{MT19}.
\medskip

On the other hand, when $p=q$ the equation \eqref{1} reduces to the following
$$\Delta_pu+\frac{1}{2}h(u,|\nabla u|^2)=0,$$
which includes well-known Lane-Emden equation and Hamilton-Jacobi equation, and has already been the subject of countless publications. One has made great progress on the local and global properties for solutions to this equation (see \cite{BGVeron, CGS, CM, GS, HHuW, HW, Lions, SerZ, Yau}). In particular, He, Wang and Wei \cite{HWW} proved recently any positive solution $u\in C^1(B_R(o))$ to the following equation $$\Delta_pu +u^l=0$$
with $$l\in\left(-\infty,~\frac{(n+3)(p-1)}{n-1}\right)$$
defined on a geodesic ball $B_R(o)$ of a complete Riemannian manifold $(M,g)$ with $\mathrm{Ric}_g\geq-(n-1)\kappa g$ satisfies
\begin{align*}
\sup_{B_{R/2}(o)} \frac{|\nabla u|}{u}  \leq C(n,p,l)\frac{1+\sqrt{\kappa}R}{R}.
\end{align*}
Shortly later, Han, He and Wang in \cite{HHW, HHW1} proved any solution $u\in C^1(B_R(o))$ to the following viscous Hamilton-Jacobi equation
$$\Delta_pu - |\nabla u|^r=0$$
with $r>p-1$ defined on a geodesic ball $B_R(o)$ of a complete Riemannian manifold $(M,g)$ with $\mathrm{Ric}_g\geq-(n-1)\kappa g$ satisfies
\begin{align*}
\sup_{B_{R/2}(o)}|\nabla u|  \leq C(n,p,r)\left(\frac{1+\sqrt{\kappa}R}{R}\right)^{\frac{1}{r-p+1}},
\end{align*}
whic improves the conclusions in \cite{BGVeron0, Lions}.
\medskip

Now, we go back to $(p,q)$-Laplace equation. If $h(u, |\nabla|^2) = a|\nabla u|^r$, where $a$ is a constant, \eqref{1} reduces to the following equation
\begin{align}\label{1.3}
\Delta_pu+\Delta_qu+a|\nabla u|^r=0,
\end{align}
which can be regarded as a natural generalization of Hamilton-Jacobi equation mentioned above. In the case $a=0$ in \eqref{1.3}, we call \eqref{1.3} as $(p,q)$-harmonic function equation.

For the above equation \eqref{1.3} we would like to ask the following two problems:
\begin{enumerate}
\item {\em A fundamental and important problem is whether or not the Liouville theorem for $(p, q)$-harmonic functions on a complete non-compact Riemannian manifold $M$ with non-negative Ricci curvature also holds true?}

\item {\em Another natural problem is whether or not similar gradient estimates for solutions to the above equation \eqref{1.3} with that of Hamilton-Jacobi equation hold true?}
\end{enumerate}

For the above problem (1), the authors of this paper do not know any result on $(p, q)$-harmonic functions except for \cite{Shen-Zhu, WZ, WYZh}, provided $p\neq q$. For the second problem, by the best knowledge of authors it seems that there is little literature.

Our goal of this paper is to answer partially the above two problems. For this end, we need to overcome some new difficulties and employ some new techniques because of $(p, q)$-Laplace operator is more complicated than $p$-Laplace operator.

\subsection{Main results}\

By a solution $u$ of \eqref{1} in an (arbitrary) domain $\Omega$, we mean a solution $u\in C^1(\Omega)\cap C^3(\widetilde{\Omega})$, where $\widetilde{\Omega}=\left\{x\in\Omega:~|\nabla u(x)|\neq 0 \right\}$. It is worth mentioning that, if the coefficients of \eqref{1} satisfy some suitable conditions, it is well-known that any solution of \eqref{1} satisfies $u\in C^{1,\alpha}(\Omega)$ for some $\alpha\in(0,1)$ (for example, see \cite{16, HMW, 17, 18}).

Now we are in the position to state our main results. Firstly, we deal with the nonnegative solutions to quasilinear elliptic differential inequality defined on a manifold $(M, g)$ with $\dim M=n\ge2$. Inspired by \cite{CM} we employ an integral estimate method to prove the Liouville theorem for $(p,q)$-harmonic function as follows.

\begin{theorem}\label{theorem1.1}
Let $(M,~g)$ be a complete noncompact Riemannian manifold with nonnegative Ricci curvature. If $u\in C^1(M)$ satisfies
\begin{align}\label{1.2}
\begin{split}
\begin{cases}
\Delta_p u+\Delta_q u\le0,&in~ M,\\[3mm]
u\ge0,& in~ M,\\
\end{cases}
\end{split}
\end{align}
where $n\le p\le q$, then $u$ is an nonnegative constant. In particular, if $u$ is a nonnegative solution to the following equation
\begin{align*}
\Delta_pu+\Delta_qu=0,
\end{align*}
then $u$ is a trivial constant solution.
\end{theorem}

Secondly, inspired by \cite{HHW, HHW1, 7, WZ} we shall combine Bochner formula, Saloff-Coste's Sobolev inequality and the Nash-Moser iteration method to study the gradient estimate and the Liouville property of the above equation \eqref{1}, defined on a complete Riemannian manifold. We can give a gradient estimate of solutions to \eqref{1} under some assumptions on $h$, which can be regarded as an extension of results in \cite{HHW}.

\begin{theorem}\label{theorem1.2}
Let $(M,~g)$ be a complete Riemannian manifold with $\mathrm{Ric}\ge-(n-1)\kappa$ ($\kappa\ge0$, $n\ge3$), and let $u$ be a solution of equation \eqref{1} in $\Omega=B(o,R)\subset M$. Assume that $h\in C^1(\mathbb{R}\times\mathbb{R}^+)$ satisfies one of the following two conditions
\begin{enumerate}
\item $\frac{\partial h}{\partial x}(x,y)\le0$,\quad $\left(\frac{\partial h}{\partial y}(x,y)\right)^2\le\mu^2 y^{r-2}$ \quad\mbox{and}\quad $h^2(x,y)\ge\lambda y^r$;

\item $\frac{\partial h}{\partial x}(x,y)\le-\lambda y^{r-\frac{q}{2}}$ \quad\mbox{and}\quad $\left(\frac{\partial h}{\partial y}(x,y)\right)^2\le\mu^2 y^{r-2}$,
\end{enumerate}
where $r>q-1\geq p-1>0$, $\lambda>0$ and $\mu>0$. Then, for any given $\delta>0$ and $\theta>1$ there exists a positive constant $\widetilde{\mathcal{C}} =\widetilde{\mathcal{C}}(n,p,q,\lambda,r,\delta,\theta)$ such that
\begin{align*}
\sup_{B_{R/4}(o)}|\nabla u|^2\le\widetilde{\mathcal{C}}\left[\left(\frac{1+\kappa R^2}{R^2}\right)^\frac{1}{r-p+1}+\left(\frac{1+\kappa R^2}{R^2}\right)^\frac{\theta}{r-q+1}\right],
\end{align*}
where $R>\delta>0$. In particular, $\widetilde{\mathcal{C}}$ does not depend on $\delta$ if $R$ is large enough.
\end{theorem}

\noindent\textbf{Remark 1.} Here we would like to give several comments on the above corollary.
\begin{enumerate}
\item When $p=q$ in \eqref{1}, it is easy to see from the arguments given in Section 4 that one can choose $\theta=1$ and $\delta=0$. In other words, we can recover the result obtained in \cite{HHW1}.

\item In fact, for the case $\dim(M) =2$ we can also obtain some similar conclusions with the case $\dim(M)\ge 3$.

\item The assumptions (1) and (2) on function $h(x, y)$ can be weaken, actually, one need only to assume $h(u(x), f(x))\equiv h(u(x), |\nabla u(x)|^2)$ satisfies pointwisely (with respect to $x$) one of the above two conditions stated in Theorem \ref{theorem1.2}. It is easy to see this from the arguments given in Section 4.
\end{enumerate}

By using the above theorem, we can easily achieve the following Liouville type theorem.

\begin{corollary}\label{corollary1.3}
Let $(M,~g)$ ($\dim(M)\ge3$) be a complete noncompact Riemannian manifold with nonnegative Ricci curvature and let $u$ be a solution of equation \eqref{1} in $M$.  Assume that $h\in C^1(\mathbb{R}\times\mathbb{R}^+)$ satisfies one of the following two conditions
\begin{enumerate}
\item $\frac{\partial h}{\partial x}(x,y)\le0$,\quad $\left(\frac{\partial h}{\partial y}(x,y)\right)^2\le\mu^2 y^{r-2}$ \quad\mbox{and}\quad $h^2(x,y)\ge\lambda y^r$;

\item $\frac{\partial h}{\partial x}(x,y)\le-\lambda y^{r-\frac{q}{2}}$ \quad\mbox{and}\quad $\left(\frac{\partial h}{\partial y}(x,y)\right)^2\le\mu^2 y^{r-2}$,
\end{enumerate}
where $r>q-1\geq p-1>0$, $\lambda>0$ and $\mu>0$. Then either $u$ is a constant or $u$ does not exist.
\end{corollary}

Now, assume that $h(u,|\nabla u|^2)=a|\nabla u|^r$ in \eqref{1}, where $a\neq0$ and $r>q-1$, in other words \eqref{1} reduces to the equation \eqref{1.3}:
\begin{align*}
\Delta_pu+\Delta_qu+a|\nabla u|^r=0.
\end{align*}
It is easy to see that $h(x, y)$ fulfills the conditions posed in Theorem \ref{theorem1.2}, hence we obtain the following theorem.

\begin{theorem}\label{theorem1.4}
Let $(M,~g)$ be a complete Riemannian manifold with $\mathrm{Ric}\ge-(n-1)\kappa$ ($\kappa\ge0$, $n\ge3$) and $1<p\le q$. If $u$ is a solution to the above equation \eqref{1.3} on $M$, where $a\neq0$ and $r>q-1$, then, for any given $\delta>0$ and $\theta>1$ there exists a positive constant $\widetilde{\mathcal{C}} =\widetilde{\mathcal{C}}(n,p,q,\lambda,r,\delta,\theta)$ such that
\begin{align*}
\sup_{B_{R/4}(o)}|\nabla u|^2\le\widetilde{\mathcal{C}}\left[\left(\frac{1+\kappa R^2}{R^2}\right)^\frac{1}{r-p+1}+\left(\frac{1+\kappa R^2}{R^2}\right)^\frac{\theta}{r-q+1}\right],
\end{align*}
where $R>\delta>0$.
\end{theorem}

Moreover, according to Corollary \ref{corollary1.3} we have the following conclusion:
\begin{corollary}\label{theorem1.4}
Let $1<p\le q$ and $M$ ($\dim(M)\ge3$) be a complete non-compact Riemannian manifold with non-negative Ricci curvature. If $u$ is a solution to the above equation \eqref{1.3} on $M$, i.e. $u$ satisfies
\begin{align*}
\Delta_pu+\Delta_qu+a|\nabla u|^r=0,
\end{align*}
where $a\neq0$ and $r>q-1$, then $u$ is a trivial constant solution.
\end{corollary}

\subsection{Organization of paper}\

Our paper is organized as follows: In Section 2, we first recall some preliminary and then establish some important lemmas, which will play a key role in the Nash-Moser iteration process. In Section 3,  we provide a simple proof of Theorem \ref{theorem1.1} by using  Bishop-Gromov volume estimate, which is inspired by \cite{CM}. We give a proof of Theorem \ref{theorem1.2} by using Nash-Moser iteration method in Section 4, which is the main body of this paper. Some further applications are presented in Section 5.

\section{\textbf{Preliminary lemmas}}
Throughout this paper, we denote $(M,g)$ an $n$-dim Riemannian manifold $(n\ge3)$, and $\nabla$ the corresponding Levi-Civita connection. We denote the volume form
$$d\mathrm{vol}=\sqrt{\det(g_{ij})}dx_1\wedge\cdots\wedge dx_n,$$
where $(x_1,\cdots,x_n)$ is a local coordinates, and for simplicity we usually omit the volume form of integral over $M$.

The $z$-Laplacian operator is defined by
\begin{center}
$
\Delta_zu=\mathrm{div}\left(\left|\nabla u\right|^{z-2}\nabla u\right),
$
\end{center}
where $z$ is a real number. Now we consider the linearized operator $\mathcal{L}_z$ of $z$-Laplace operator:
\begin{align}\label{2.1}
\mathcal{L}_z(\psi)=\mathrm{div}\left(f^{\frac{z}{2}-1}A_z\left(\nabla\psi\right)\right),
\end{align}
where $f=\left|\nabla u\right|^2$, and
\begin{align}\label{2.2}
A_z\left(\nabla\psi\right)=\nabla\psi+(z-2)f^{-1}\langle\nabla\psi,\nabla u\rangle\nabla u.
\end{align}
We first derive an useful expression of $\mathcal{L}_z(f)$.

\begin{lemma}\label{lemma2.1}
The equality
\begin{align*}
\mathcal{L}_z(f)=\left(\dfrac{z}{2}-1\right)f^{\frac{z}{2}-2}\left|\nabla f\right|^2+2f^{\frac{z}{2}-1}\left(\left|\nabla\nabla u\right|^2+\mathrm{Ric}\left(\nabla u,\nabla u\right)\right)+2\langle\nabla\Delta_zu,\nabla u\rangle
\end{align*}
holds point-wisely in $\{ x\in M:~f(x)>0$$\}$.
\end{lemma}

\begin{proof}
By \eqref{2.2}, we have
\begin{align}\label{2.3}
 A_z\left(\nabla f\right)=\nabla f+(z-2)f^{-1}\langle\nabla f,\nabla u\rangle\nabla u.
\end{align}
Combining \eqref{2.1} and \eqref{2.3} together, we obtain
\begin{align}
\begin{split}\label{2.4}
\mathcal{L}_z(f)=&\left(\dfrac{z}{2}-1\right)f^{\frac{z}{2}-2}|\nabla f|^2+f^{\frac{z}{2}-1}\Delta f+(z-2)\left(\dfrac{z}{2}-2\right)f^{\frac{z}{2}-3}\langle\nabla f,\nabla u\rangle^2\\
&+(z-2)f^{\frac{z}{2}-2}\langle\nabla\langle\nabla f,\nabla u\rangle,\nabla u\rangle+(z-2)f^{\frac{z}{2}-2}\langle\nabla f,\nabla u\rangle\Delta u.
\end{split}
\end{align}
On the other hand, by the definition of the $z$-Laplacian, we have
\begin{align}
\begin{split}\label{2.5}
2\langle\nabla \Delta_zu,\nabla u\rangle=&(z-2)\left(\dfrac{z}{2}-2\right)f^{\frac{z}{2}-3}\langle\nabla f,\nabla u\rangle^2+(z-2)f^{\frac{z}{2}-2}\langle\nabla\langle\nabla f,\nabla u\rangle,\nabla u\rangle\\
&+(z-2)f^{\frac{z}{2}-2}\langle\nabla f,\nabla u\rangle\Delta u+2f^{\frac{z}{2}-1}\langle\nabla \Delta u,\nabla u\rangle.
\end{split}
\end{align}
Combining \eqref{2.4} and \eqref{2.5} together, we obtain
\begin{align}
\begin{split}\label{3.7}
\mathcal{L}_z(f)=\left(\dfrac{z}{2}-1\right)f^{\frac{z}{2}-2}|\nabla f|^2+f^{\frac{z}{2}-1}\Delta f+2\langle\nabla \Delta_zu,\nabla u\rangle-2f^{\frac{z}{2}-1}\langle\nabla \Delta u,\nabla u\rangle.
\end{split}
\end{align}
By \eqref{3.7} and the following Bochner formula
\begin{center}
 $
 \frac{1}{2}\Delta f=\left|\nabla\nabla u\right|^2+\mathrm{Ric}\left(\nabla u,\nabla u\right)+\langle\nabla \Delta u,\nabla u\rangle,
 $
 \end{center}
we finish the proof of Lemma \ref{lemma2.1}.
\end{proof}

Another tool that will be used in the later is the following lemma (for the proof see \cite{15}):
\begin{lemma}\label{lemma2.2}
Suppose that $\phi(t)$ is a positive and bounded function, which is defined on $[T_0,T_1]$. If for all $T_0\le t<s\le T_1$, $\phi$ satisfies
    \begin{center}
        $
        \phi(t)\le\theta_0\phi(s)+\dfrac{A}{(s-t)^{\alpha_0}}+B,
        $
    \end{center}
where $\theta_0<1$, $A$, $B$ and $\alpha_0$ are some non-negative constants. Then, for any $T_0\le \rho<\tau\le T_1$ there exists
\begin{center}
$
\phi(\rho)\le C(\alpha_0,\theta_0)\left[\dfrac{A}{(\tau-\rho)^{\alpha_0}}+B\right],
$
\end{center}
where $C(\alpha_0,\theta_0)$ is a positive constant which depends only on $\alpha_0$ and $\theta_0$. Furthermore, if we set $\theta_0=\dfrac{1}{2}$ and let $\alpha_0$ change in a bounded interval, then there exists a positive constant $C_0$ such that $C(\alpha_0,\frac{1}{2})\le C_0$.
\end{lemma}

We finally recall the following lemma, which shall play a key role in our proof of the main theorems.
\begin{lemma}(Saloff-Coste \cite{14})
Let $(M,g)$ be a complete manifold with $Ric\ge-(n-1)\kappa$. For $n>2$, there exists a positive constant $C_n$ depending only on $n$, such that for all $B\subset M$ of radius $R$ and volume $V$ we have for $h\in C_0^\infty(B)$
\begin{align*}
\left\|h\right\|^2_{L^{\frac{2n}{n-2}}(B)}\le \exp\left\{C_n(1+\sqrt{\kappa}R)\right\}V^{-\frac{2}{n}}R^2\left(\int_B{\left|\nabla h\right|^2+R^{-2}h^2}\right).
\end{align*}
For $n=2$, the above inequality holds with n replaced by any fixed $n'>2$.
\end{lemma}

\section{\textbf{Proof of Theorem \ref{theorem1.1}}}

In this section we do not adopted the Moser iteration method and use instead an integral estimate method to give a simple proof of Theorem \ref{theorem1.1}.
\medskip

\noindent\textbf{Proof of Theorem \ref{theorem1.1}}
\begin{proof}
Let $\eta\in C_0^{\infty}(M)$ be a nonnegative function.  Multiplying \eqref{1.2} by $(u+1)^{1-q}\eta^q$ and integrating over $M$, we obtain the following
\begin{align*}
\int_M\left(\Delta_p u+\Delta_q u\right)(u+1)^{1-q}\eta^q\le0.
\end{align*}
Integrating by parts, we obtain
\begin{align*}
-\int_M\left\langle f^{\frac{p}{2}-1}\nabla u+f^{\frac{q}{2}-1}\nabla u,\,\nabla\left[(u+1)^{1-q}\eta^q\right]\right\rangle\le0.
\end{align*}
Hence, we have
\begin{align}\label{3.1}
\begin{split}
(q-1)\int_M\left(f^{\frac{p}{2}}+f^{\frac{q}{2}}\right)(u+1)^{-q}\eta^q\le q\int_M\left(f^{\frac{p}{2}-1}+f^{\frac{q}{2}-1}\right)(u+1)^{1-q}\eta^{q-1}\langle\nabla u,\nabla\eta\rangle.
\end{split}
\end{align}
By absolute value inequality and Cauchy inequality, we obtain
\begin{align*}
&q\int_M\left(f^{\frac{p}{2}-1}+f^{\frac{q}{2}-1}\right)(u+1)^{1-q}\eta^{q-1}\langle\nabla u,\nabla\eta\rangle\\
\le&q\int_M\left(f^{\frac{p}{2}-\frac{1}{2}}+f^{\frac{q}{2}-\frac{1}{2}}\right)(u+1)^{1-q}\eta^{q-1}|\nabla\eta|\\
\le&\frac{q-1}{2}\int_M\left(f^{\frac{p}{2}}+f^{\frac{q}{2}}\right)(u+1)^{-q}\eta^q+\frac{q^2}{2(q-1)}\int_M\left(f^{\frac{p}{2}-1}+f^{\frac{q}{2}-1}\right)(u+1)^{2-q}\eta^{q-2}|\nabla\eta|^2.
\end{align*}
Substituting the above inequality into \eqref{3.1}, we have
\begin{align}\label{3.2}
\frac{(q-1)^2}{q^2}\int_M\left(f^{\frac{p}{2}}+f^{\frac{q}{2}}\right)(u+1)^{-q}\eta^q\le\int_M\left(f^{\frac{p}{2}-1}+f^{\frac{q}{2}-1}\right)(u+1)^{2-q}\eta^{q-2}|\nabla\eta|^2.
\end{align}
By using Young inequality and $p>2$, we have
\begin{align*}
\int_M f^{\frac{p}{2}-1}(u+1)^{2-q}\eta^{q-2}|\nabla\eta|^2\le \frac{(q-1)^2}{2q^2}\int_Mf^{\frac{p}{2}}(u+1)^{-q}\eta^q+\mathcal{C}(p,q)\int_M(u+1)^{p-q}\eta^{q-p}|\nabla \eta|^p
\end{align*}
and
\begin{align*}
\int_M f^{\frac{q}{2}-1}(u+1)^{2-q}\eta^{q-2}|\nabla\eta|^2\le \frac{(q-1)^2}{2q^2}\int_Mf^{\frac{q}{2}}(u+1)^{-q}\eta^q+\mathcal{C}(q)\int_M|\nabla \eta|^q.
\end{align*}
Substituting the above inequalities into \eqref{3.2}, we have
\begin{align*}
\frac{(q-1)^2}{2q^2}\int_M\left(f^{\frac{p}{2}}+f^{\frac{q}{2}}\right)(u+1)^{-q}\eta^q\le\mathcal{C}(p,q)\int_M(u+1)^{p-q}\eta^{q-p}|\nabla \eta|^p+\mathcal{C}(q)\int_M|\nabla \eta|^q.
\end{align*}
Furthermore, by using \eqref{3.2}, we can know that the above inequality is also true under $p=2$. Since $p\le q$ and $u\ge 0$, we arrive at
\begin{align*}
\frac{(q-1)^2}{2q^2}\int_M\left(f^{\frac{p}{2}}+f^{\frac{q}{2}}\right)(u+1)^{-q}\eta^q\le\mathcal{C}(p,q)\int_M\eta^{q-p}|\nabla \eta|^p+\mathcal{C}(q)\int_M|\nabla \eta|^q.
\end{align*}

Let $\Phi\in C^\infty(\mathbb{R})$ be a nonnegative function such that $0\le\Phi\le1$, $\Phi'\le0$, $\Phi=1$ on $(-\infty,1]$, $\Phi=0$ on $[2,+\infty)$ and $|\Phi'|\le C$ for some positive constant $C$. For any $\rho>1$, define
$$\Psi=\Phi\left(\frac{\log r_0}{\log\rho}\right),$$
where $r_0$ is the distance function. Then, we have $\Psi\equiv 1$ on $B_\rho$, $\Psi\equiv 0$ on $B_{\rho^2}^c$ and
$$|\nabla \Psi|\le\frac{C}{r_0\log\rho}$$
on $B_{\rho^2}\setminus B_\rho$. With the choice of $\eta=\Psi$ and the above integral inequality yield
\begin{align}\label{3.3}
\frac{(q-1)^2}{2q^2}\int_{B_\rho}\left(f^{\frac{p}{2}}+f^{\frac{q}{2}}\right)(u+1)^{-q}\le\frac{\mathcal{C}}{(\log \rho)^p}\int_{B_{\rho^2}\setminus B_\rho}\frac{1}{{r_0}^p}+\frac{\mathcal{C}}{(\log \rho)^q}\int_{B_{\rho^2}\setminus B_\rho}\frac{1}{{r_0}^q}.
\end{align}

Let $V(r_0)=\mathrm{vol}(B_{r_0})$ and $S(r_0)=\mathrm{meas}(\partial B_{r_0})$ denot the area of $B_{r_0}$ and the area of $\partial B_{r_0}$ respectively. Then, by using the coarea formula, i.e. $V'=S$, an integration by parts and the Bishop-Gromov volume comparison theorem, we obtain
\begin{align*}
\begin{split}
\int_{B_{\rho^2}\setminus B_\rho}\frac{1}{{r_0}^z}=\int_\rho^{\rho^2}\frac{1}{{r_0}^z}S(r_0)dr_0= &\left[\frac{1}{{r_0}^z}V(r_0)\right]_\rho^{\rho^2}+z\int_\rho^{\rho^2}\frac{1}{{r_0}^{z+1}}V(r_0)dr_0\\
\le&\mathcal{C}\left(\frac{1}{\rho^{z-n}}+\int_\rho^{\rho^2}\frac{1}{r_0^{z+1-n}}dr\right)\\
\le&\begin{cases}
\mathcal{C}(1+\log\rho),&z=n,\\[3mm]
\dfrac{\mathcal{C}}{\rho^{z-n}},&z>n.
\end{cases}
\end{split}
\end{align*}
Hence, we arrive at
\begin{align*}
    \dfrac{\mathcal{C}}{(\log\rho)^z}\int_{B_{\rho^2}\setminus B_\rho}\frac{1}{{r_0}^z}\le\begin{cases}
\dfrac{\mathcal{C}(1+\log\rho)}{(\log\rho)^n},&z=n,\\[3mm]
\dfrac{\mathcal{C}}{\rho^{z-n}(\log\rho)^z},&z>n.
\end{cases}
\end{align*}
Therefore, we can deduce that
\begin{align}\label{3.4}
\lim_{\rho\to+\infty}\dfrac{\mathcal{C}}{(\log\rho)^z}\int_{B_{\rho^2}\setminus B_\rho}\frac{1}{{r_0}^z}=0,
\end{align}
for any $z\ge n$. Combining \eqref{3.3} and \eqref{3.4} together, we arrive at
\begin{align*}
\int_M\left(f^{\frac{p}{2}}+f^{\frac{q}{2}}\right)(u+1)^{-q}=0,
\end{align*}
which implies that
\begin{align*}
|\nabla u|\equiv 0\quad\mbox{in}~M.
\end{align*}
Combining above, we finish the proof of Theorem \ref{theorem1.1}.
\end{proof}

     %Let $0\le\eta\le1$, $\eta\equiv 1$ on $B_{\frac{R}{2}}$, $\eta\equiv 0$ on $B_{R}^c$ and $|\nabla \eta|\le\frac{C}{R}$ on $B_{R}\setminus B_{\frac{R}{2}}$, we obtain
      %\begin{align*}
         %\int_{B_{\frac{R}{2}}}\left(f^{\frac{p}{2}}+f^{\frac{q}{2}}\right)(u+1)^{-q}\le\mathcal{C}\int_{B_R}\left(\frac{1}{R^p}+\frac{1}{R^q}\right).
     %\end{align*}
     %By the Bishop-Gromov volume estimate, we conclude
         %\begin{align*}
         %\int_{B_{\frac{R}{2}}}\left(f^{\frac{p}{2}}+f^{\frac{q}{2}}\right)(u+1)^{-q}\le\mathcal{C}\left(\frac{1}{R^{p-n}}+\frac{1}{R^{q-n}}\right).
     %\end{align*}
     %Therefore, by letting $R\longrightarrow+\infty$, we have
     %\begin{align*}
          %\int_M\left(f^{\frac{p}{2}}+f^{\frac{q}{2}}\right)(u+1)^{-q}=0.
     %\end{align*}
     %Hence, we have
     %\begin{align*}
         %\left(f^{\frac{p}{2}}+f^{\frac{q}{2}}\right)(u+1)^{-q}\equiv0\quad in~M,
     %\end{align*}
     %which implies that
     %\begin{align*}
         %|\nabla u|\equiv 0\quad in~M.
     %\end{align*}
     %Combining above, we finish the proof of Theorem \ref{theorem1.1}.

\section{\textbf{Proof of Theorem \ref{theorem1.2}}}
First, we need to give the point-wise estimates of $\mathcal{L}_{p,q}(f)$, where $\mathcal{L}_{p,q}$ is the linearized operator of $\Delta_p +\Delta_q$ (the sum of $p$-Laplacian and $q$-Laplacian) at $u$.

\subsection{\textbf{Estimate for the linearized operator of $p$-Laplace $+$ $q$-Laplace operator}}\
\begin{lemma}\label{lemma4.1} Let $u$ be a solution of equation \eqref{1} in $\Omega\subset M$. Denote $f=\left|\nabla u\right|^2$. Then, the following
\begin{align*}
\mathcal{L}_{p,q}(f)=&\left(\frac{p}{2}-1\right)f^{\frac{p}{2}-2}|\nabla f|^2+\left(\frac{q}{2}-1\right)f^{\frac{q}{2}-2}|\nabla f|^2+2\left(f^{\frac{p}{2}-1}+f^{\frac{q}{2}-1}\right)\left(\left|\nabla\nabla u\right|^2+\mathrm{Ric}\left(\nabla u,\nabla u\right)\right)\\
&-2\frac{\partial h}{\partial x}(u,f)f-2\frac{\partial h}{\partial y}(u,f)\langle\nabla f,\nabla u\rangle
\end{align*}
holds point-wisely in $\{x\in \Omega:~f(x)>0$$\}$.
\end{lemma}
\begin{proof}
    By Lemma \ref{lemma2.1}, we arrive at
    \begin{align*}
        \mathcal{L}_{p,q}(f)=&\mathcal{L}_{p}(f)+\mathcal{L}_{q}(f)\\
        =&\left(\frac{p}{2}-1\right)f^{\frac{p}{2}-2}|\nabla f|^2+\left(\frac{q}{2}-1\right)f^{\frac{q}{2}-2}|\nabla f|^2+2\left(f^{\frac{p}{2}-1}+f^{\frac{q}{2}-1}\right)\left(\left|\nabla\nabla u\right|^2+\mathrm{Ric}\left(\nabla u,\nabla u\right)\right)\\
        &+2\langle\nabla\left(\Delta_pu+\Delta_qu\right),\nabla u\rangle.
    \end{align*}
    On the other hand, by \eqref{1}, we have
\begin{align*}
\Delta_pu+\Delta_qu=-h(u,f).
\end{align*}
Combining above, we finish the proof of Lemma \ref{lemma4.1}.

\end{proof}

Using Lemma \ref{lemma4.1}, we can establish the following Lemma:

\begin{lemma}\label{lemma4.2}
    Let $(M,~g)$ be a complete Riemannian manifold with $\mathrm{Ric}\ge-(n-1)\kappa$ ($\kappa\ge0$), and let $u$ be a solution of equation \eqref{1} in $\Omega\subset M$. Assume that $h\in C^1(\mathbb{R}\times\mathbb{R}^+)$ satisfies
    \begin{align*}
        \left(\frac{\partial h}{\partial y}(x,y)\right)^2\le\mu^2 y^{r-2},
    \end{align*}
    where $r>q-1$ and $\mu>0$. Then, the following
    \begin{align}\label{4.1}
       \begin{split}
            \left(f^\frac{p}{2}+f^\frac{q}{2}\right)\mathcal{L}_{p,q}(f)\ge&-2(n-1)\kappa\left(f^\frac{p}{2}+f^\frac{q}{2}\right)^2+\frac{2}{n-1}fh^2(u,f)-2\mu|\nabla f|f^{\frac{r}{2}-\frac{1}{2}} \left(f^\frac{p}{2}+f^\frac{q}{2}\right)\\
        &-\frac{2(q-1)}{n-1}|\nabla f|f^{-\frac{1}{2}}\left(f^\frac{p}{2}+f^\frac{q}{2}\right)|h(u,f)|-2\frac{\partial h}{\partial x}(u,f)f\left(f^\frac{p}{2}+f^\frac{q}{2}\right)
       \end{split}
    \end{align}
    holds point-wisely in $\{x\in \Omega:~f(x)>0$$\}$.
\end{lemma}
\begin{proof}
    Let $\{e_1,\cdots,e_n\}$ be an orthonormal frame of $TM^n$ on a domain with $f>0$ such that $e_1=\frac{\nabla u}{|\nabla u|}$. We hence infer that
\begin{align}\label{4.2}
\sum_{i=1}^nu_{1i}^2=\frac{|\nabla f|^2}{4f}\quad and\quad u_{11}=\frac{\langle\nabla f,\nabla u\rangle}{2f}
\end{align}
and
\begin{align}\label{4.3}
    \Delta_zu=f^{\frac{z}{2}-1}\left[(z-1)u_{11}+\sum_{i=2}^nu_{ii}\right].
\end{align}
By omitting some non-negative terms in $\left|\nabla\nabla u\right|^2$, we obtain
\begin{align*}
\left|\nabla\nabla u\right|^2\ge\sum_{i=1}^nu^2_{1i}+\sum_{i=2}^nu^2_{ii}.
\end{align*}
By \eqref{4.2} and Cauchy inequality, we have
\begin{align*}
    \left|\nabla\nabla u\right|^2\ge\frac{|\nabla f|^2}{4f}+\frac{1}{n-1}\left(\sum_{i=2}^nu_{ii}\right)^2.
\end{align*}
By using Lemma \ref{lemma4.1} and the above inequality, we have
\begin{align*}
    \mathcal{L}_{p,q}(f)\ge&\left(\frac{p}{2}-\frac{1}{2}\right)f^{\frac{p}{2}-2}|\nabla f|^2+\left(\frac{q}{2}-\frac{1}{2}\right)f^{\frac{q}{2}-2}|\nabla f|^2+2\left(f^{\frac{p}{2}-1}+f^{\frac{q}{2}-1}\right)\mathrm{Ric}\left(\nabla u,\nabla u\right)\\
&+\frac{2}{n-1}\left(f^{\frac{p}{2}-1}+f^{\frac{q}{2}-1}\right)\left(\sum_{i=2}^nu_{ii}\right)^2-2\frac{\partial h}{\partial x}(u,f)f-2\frac{\partial h}{\partial y}(u,f)\langle\nabla f,\nabla u\rangle.
\end{align*}
Since $\mathrm{Ric}\ge-(n-1)\kappa$ and $1<p\le q$, we arrive at
\begin{align}\label{4.4}
    \begin{split}
        \mathcal{L}_{p,q}(f)\ge&-2(n-1)\kappa\left(f^{\frac{p}{2}}+f^{\frac{q}{2}}\right)+\frac{2}{n-1}\left(f^{\frac{p}{2}-1}+f^{\frac{q}{2}-1}\right)\left(\sum_{i=2}^nu_{ii}\right)^2\\
&-2\frac{\partial h}{\partial x}(u,f)f-2\frac{\partial h}{\partial y}(u,f)\langle\nabla f,\nabla u\rangle.
    \end{split}
\end{align}
By \eqref{1} and \eqref{4.3}, we have
\begin{align*}
    \sum_{i=2}^nu_{ii}=-\left(f^{\frac{p}{2}-1}+f^{\frac{q}{2}-1}\right)^{-1}\left[(p-1)f^{\frac{p}{2}-1}u_{11}+(q-1)f^{\frac{q}{2}-1}u_{11}+h(u,f)\right].
\end{align*}
Hence, we obtain
\begin{align*}
    &\frac{2}{n-1}\left(f^{\frac{p}{2}-1}+f^{\frac{q}{2}-1}\right)\left(\sum_{i=2}^nu_{ii}\right)^2\\
    \ge&\frac{2}{n-1}\left(f^{\frac{p}{2}-1}+f^{\frac{q}{2}-1}\right)^{-1}\left\{h^2(u,f)+2h(u,f)\left[(p-1)f^{\frac{p}{2}-1}+(q-1)f^{\frac{q}{2}-1}\right]u_{11}\right\}\\
    \ge&\frac{2}{n-1}\left(f^{\frac{p}{2}-1}+f^{\frac{q}{2}-1}\right)^{-1}h^2(u,f)-\frac{4(q-1)}{n-1}|u_{11}||h(u,f)|.
\end{align*}
By \eqref{4.2}, we have
\begin{align*}
    \frac{2}{n-1}\left(f^{\frac{p}{2}-1}+f^{\frac{q}{2}-1}\right)\left(\sum_{i=2}^nu_{ii}\right)^2\ge\frac{2}{n-1}\left(f^{\frac{p}{2}-1}+f^{\frac{q}{2}-1}\right)^{-1}h^2(u,f)-\frac{2(q-1)}{n-1}f^{-\frac{1}{2}}|\nabla f||h(u,f)|.
\end{align*}
Substituting the above inequality, into \eqref{4.4}, we obtain
\begin{align}\label{4.5}
     \begin{split}
         \mathcal{L}_{p,q}(f)\ge&-2(n-1)\kappa\left(f^{\frac{p}{2}}+f^{\frac{q}{2}}\right)+\frac{2}{n-1}\left(f^{\frac{p}{2}-1}+f^{\frac{q}{2}-1}\right)^{-1}h^2(u,f)\\
     &-\frac{2(q-1)}{n-1}f^{-\frac{1}{2}}|\nabla f||h(u,f)|-2\frac{\partial h}{\partial x}(u,f)f-2\frac{\partial h}{\partial y}(u,f)\langle\nabla f,\nabla u\rangle.
     \end{split}
\end{align}
By using the assumptions of $h$, we arrive at
\begin{align*}
    -2\frac{\partial h}{\partial y}(u,f)\langle\nabla f,\nabla u\rangle\ge-2\left|\frac{\partial h}{\partial y}(u,f)\right|f^{\frac{1}{2}}|\nabla f|\ge-2\mu f^{\frac{r}{2}-\frac{1}{2}}|\nabla f|.
\end{align*}
Substituting the above inequality into \eqref{4.5}, we finish the proof of Lemma \ref{lemma4.2}.

\end{proof}

\subsection{Deducing the main integral inequality.}\

Now we choose a geodesic ball $\Omega=B_R(o)\subset M$. If we choose $\psi=f_\varepsilon^t\eta^2$ as a test function of \eqref{4.1}, where $\eta\in C_0^\infty(\Omega)$ is non-negative, $f_\varepsilon=(f-\varepsilon)^+$, $\varepsilon>0$, $t>1$ is to be determined later. It follows from \eqref{4.1} that
\begin{align}\label{4.6}
\begin{split}
&-\int_{\Omega}\left\langle f^{\frac{p}{2}-1}\nabla f+(p-2)f^{\frac{p}{2}-2}\langle\nabla f,\nabla u\rangle\nabla u,\nabla\left[f_\varepsilon^t\eta^2\left(f^\frac{p}{2}+f^\frac{q}{2}\right)\right]\right\rangle\\
&-\int_{\Omega}\left\langle f^{\frac{q}{2}-1}\nabla f+(q-2)f^{\frac{q}{2}-2}\langle\nabla f,\nabla u\rangle\nabla u,\nabla\left[f_\varepsilon^t\eta^2\left(f^\frac{p}{2}+f^\frac{q}{2}\right)\right]\right\rangle\\
\ge&-2(n-1)\kappa\int_{\Omega}\left(f^\frac{p}{2}+f^\frac{q}{2}\right)^2f_\varepsilon^t\eta^2+\frac{2}{n-1}\int_\Omega fh^2(u,f)f_\varepsilon^t\eta^2\\
&-2\int_\Omega\frac{\partial h}{\partial x}(u,f)f\left(f^\frac{p}{2}+f^\frac{q}{2}\right)f_\varepsilon^t\eta^2 -2\mu\int_\Omega|\nabla f|f^{\frac{r}{2}-\frac{1}{2}} \left(f^\frac{p}{2}+f^\frac{q}{2}\right)f_\varepsilon^t\eta^2\\
&-\frac{2(q-1)}{n-1}\int_\Omega |\nabla f|f^{-\frac{1}{2}}\left(f^\frac{p}{2}+f^\frac{q}{2}\right)|h(u,f)|f_\varepsilon^t\eta^2.
\end{split}
\end{align}
Direct computation shows that
\begin{align*}
    &-\int_{\Omega}\left\langle f^{\frac{z}{2}-1}\nabla f+(z-2)f^{\frac{z}{2}-2}\langle\nabla f,\nabla u\rangle\nabla u,\nabla\left[f_\varepsilon^t\eta^2\left(f^\frac{p}{2}+f^\frac{q}{2}\right)\right]\right\rangle\\
    =&-\int_\Omega\left(\frac{p}{2}f^{\frac{p+z}{2}-2}+\frac{q}{2}f^{\frac{q+z}{2}-2}\right)f_\varepsilon^t|\nabla f|^2\eta^2-(z-2)\int_\Omega\left(\frac{p}{2}f^{\frac{p+z}{2}-3}+\frac{q}{2}f^{\frac{q+z}{2}-3}\right)f_\varepsilon^t\langle\nabla f,\nabla u\rangle^2\eta^2\\
    &-t\int_\Omega\left(f^{\frac{p+z}{2}-1}+f^{\frac{q+z}{2}-1}\right)f_\varepsilon^{t-1}|\nabla f|^2\eta^2-(z-2)t\int_\Omega\left(f^{\frac{p+z}{2}-2}+f^{\frac{q+z}{2}-2}\right)f_\varepsilon^{t-1}\langle\nabla f,\nabla u\rangle^2\eta^2\\
    &-2\int_\Omega\left(f^{\frac{p+z}{2}-1}+f^{\frac{q+z}{2}-1}\right)f_\varepsilon^{t}\langle\nabla f,\nabla \eta\rangle\eta-2(z-2)\int_\Omega\left(f^{\frac{p+z}{2}-2}+f^{\frac{q+z}{2}-2}\right)f_\varepsilon^{t}\langle\nabla f,\nabla u\rangle\langle\nabla u,\nabla\eta\rangle\eta.
\end{align*}
By using absolute value inequality and $[(2-z)_+-1]<0$ ($z\in\{p,q\}$), we arrive at
\begin{align*}
    &-\int_{\Omega}\left\langle f^{\frac{z}{2}-1}\nabla f+(z-2)f^{\frac{z}{2}-2}\langle\nabla f,\nabla u\rangle\nabla u,\nabla\left[f_\varepsilon^t\eta^2\left(f^\frac{p}{2}+f^\frac{q}{2}\right)\right]\right\rangle\\
    \le&[(2-z)_+-1]\int_\Omega\left(\frac{p}{2}f^{\frac{p+z}{2}-2}+\frac{q}{2}f^{\frac{q+z}{2}-2}\right)f_\varepsilon^t|\nabla f|^2\eta^2\\
    &+[(2-z)_+-1]t\int_\Omega\left(f^{\frac{p+z}{2}-1}+f^{\frac{q+z}{2}-1}\right)f_\varepsilon^{t-1}|\nabla f|^2\eta^2\\
    &+2(1+|z-2|)\int_\Omega\left(f^{\frac{p+z}{2}-1}+f^{\frac{q+z}{2}-1}\right)f_\varepsilon^{t}|\nabla f||\nabla \eta|\eta\\
    \le&[(2-z)_+-1]t\int_\Omega\left(f^{\frac{p+z}{2}-1}+f^{\frac{q+z}{2}-1}\right)f_\varepsilon^{t-1}|\nabla f|^2\eta^2\\
    &+2(1+|z-2|)\int_\Omega\left(f^{\frac{p+z}{2}-1}+f^{\frac{q+z}{2}-1}\right)f_\varepsilon^{t}|\nabla f||\nabla \eta|\eta.
\end{align*}
Substituting the above inequality into \eqref{4.6}, we obtain
\begin{align}\label{4.7}
    &\frac{2}{n-1}\int_\Omega fh^2(u,f)f_\varepsilon^t\eta^2+[1-(2-p)_+]t\int_\Omega\left(f^{p-1}+f^{\frac{p+q}{2}-1}\right)f_\varepsilon^{t-1}|\nabla f|^2\eta^2\nonumber\\
    &+[1-(2-q)_+]t\int_\Omega\left(f^{\frac{p+q}{2}-1}+f^{q-1}\right)f_\varepsilon^{t-1}|\nabla f|^2\eta^2-2\int_\Omega\frac{\partial h}{\partial x}(u,f)f\left(f^\frac{p}{2}+f^\frac{q}{2}\right)f_\varepsilon^t\eta^2\nonumber\\
    \le&2(n-1)\kappa\int_{\Omega}\left(f^\frac{p}{2}+f^\frac{q}{2}\right)^2f_\varepsilon^t\eta^2+2\mu\int_\Omega|\nabla f|f^{\frac{r}{2}-\frac{1}{2}} \left(f^\frac{p}{2}+f^\frac{q}{2}\right)f_\varepsilon^t\eta^2\\
    &+\frac{2(q-1)}{n-1}\int_\Omega |\nabla f|f^{-\frac{1}{2}}\left(f^\frac{p}{2}+f^\frac{q}{2}\right)|h(u,f)|f_\varepsilon^t\eta^2\nonumber\\
    &+2(1+|p-2|)\int_\Omega\left(f^{p-1}+f^{\frac{p+q}{2}-1}\right)f_\varepsilon^{t}|\nabla f||\nabla \eta|\eta\nonumber\\
    &+2(1+|q-2|)\int_\Omega\left(f^{\frac{p+q}{2}-1}+f^{q-1}\right)f_\varepsilon^{t}|\nabla f||\nabla \eta|\eta.\nonumber
\end{align}
By Cauchy inequality, we have
\begin{align*}
    &2(1+|p-2|)\int_\Omega\left(f^{p-1}+f^{\frac{p+q}{2}-1}\right)f_\varepsilon^{t}|\nabla f||\nabla \eta|\eta\\
    &+2(1+|q-2|)\int_\Omega\left(f^{\frac{p+q}{2}-1}+f^{q-1}\right)f_\varepsilon^{t}|\nabla f||\nabla \eta|\eta\\
    \le&(4+3|p-2|+|q-2|)\int_\Omega f^{p-1}f_\varepsilon^{t}|\nabla f||\nabla \eta|\eta+(4+|p-2|+3|q-2|)\int_\Omega f^{q-1}f_\varepsilon^{t}|\nabla f||\nabla \eta|\eta\\
    \le&\frac{[1-(2-p)_+]t}{2}\int_\Omega f^{p-1}f_\varepsilon^{t-1}|\nabla f|^2\eta^2+\frac{(4+3|p-2|+|q-2|)^2}{2[1-(2-p)_+]t}\int_\Omega f^{p-1}f_\varepsilon^{t+1}|\nabla \eta|^2\\
    &+\frac{[1-(2-q)_+]t}{2}\int_\Omega f^{q-1}f_\varepsilon^{t-1}|\nabla f|^2\eta^2+\frac{(4+3|q-2|+|p-2|)^2}{2[1-(2-q)_+]t}\int_\Omega f^{q-1}f_\varepsilon^{t+1}|\nabla \eta|^2
\end{align*}
and
\begin{align*}
    &\frac{2(q-1)}{n-1}\int_\Omega |\nabla f|f^{-\frac{1}{2}}\left(f^\frac{p}{2}+f^\frac{q}{2}\right)|h(u,f)|f_\varepsilon^t\eta^2\\
    \le&\frac{1}{n-1}\int_\Omega fh^2(u,f)f_\varepsilon^t\eta^2+\frac{2(q-1)^2}{n-1}\int_\Omega (f^p+f^q)|\nabla f|^2f^{-2}f_\varepsilon^t\eta^2.
\end{align*}
Substituting the above inequalities into \eqref{4.7} and omitting some nonnegative terms, we arrive at
\begin{align}\label{4.8}
    &\frac{1}{n-1}\int_\Omega fh^2(u,f)f_\varepsilon^t\eta^2+\frac{[1-(2-p)_+]t}{2}\int_\Omega f^{p-1}f_\varepsilon^{t-1}|\nabla f|^2\eta^2\nonumber\\
    &+\frac{[1-(2-q)_+]t}{2}\int_\Omega f^{q-1}f_\varepsilon^{t-1}|\nabla f|^2\eta^2-2\int_\Omega\frac{\partial h}{\partial x}(u,f)f\left(f^\frac{p}{2}+f^\frac{q}{2}\right)f_\varepsilon^t\eta^2\nonumber\\
    \le&2(n-1)\kappa\int_{\Omega}\left(f^\frac{p}{2}+f^\frac{q}{2}\right)^2f_\varepsilon^t\eta^2+2\mu\int_\Omega|\nabla f|f^{\frac{r}{2}-\frac{1}{2}} \left(f^\frac{p}{2}+f^\frac{q}{2}\right)f_\varepsilon^t\eta^2\\
    &+\frac{(4+3|p-2|+|q-2|)^2}{2[1-(2-p)_+]t}\int_\Omega f^{p-1}f_\varepsilon^{t+1}|\nabla \eta|^2+\frac{(4+3|q-2|+|p-2|)^2}{2[1-(2-q)_+]t}\int_\Omega f^{q-1}f_\varepsilon^{t+1}|\nabla \eta|^2\nonumber\\
    &+\frac{2(q-1)^2}{n-1}\int_\Omega (f^p+f^q)|\nabla f|^2f^{-2}f_\varepsilon^t\eta^2.\nonumber
\end{align}

\textbf{Case 1}: Suppose that
\begin{align*}
\frac{\partial h}{\partial x}(x,y)\le0\quad \mbox{and}\quad h^2(x,y)\ge\lambda y^r,\quad(\lambda>0)
\end{align*}
then we can achieve that
\begin{align}\label{siu1}
-2\int_\Omega\frac{\partial h}{\partial x}(u,f)f\left(f^\frac{p}{2}+f^\frac{q}{2}\right)f_\varepsilon^t\eta^2\ge0
\end{align}
and
\begin{align}\label{4.9}
\frac{1}{n-1}\int_\Omega fh^2(u,f)f_\varepsilon^t\eta^2\ge\frac{\lambda}{n-1}\int_\Omega f^{r+1}f_\varepsilon^t\eta^2.
\end{align}
Moreover, by Cauchy inequality we have
\begin{align}\label{4.10}
\begin{split}
&2\mu\int_\Omega|\nabla f|f^{\frac{r}{2}-\frac{1}{2}} \left(f^\frac{p}{2}+f^\frac{q}{2}\right)f_\varepsilon^t\eta^2\\
\le&\frac{\lambda}{2(n-1)}\int_\Omega f^{r+1}f_\varepsilon^t\eta^2+\frac{4(n-1)\mu^2}{\lambda}\int_\Omega (f^p+f^q)f^{-2}f^t_\varepsilon
|\nabla f|^2\eta^2.
\end{split}
\end{align}
Now, substituting \eqref{siu1}, \eqref{4.9} and \eqref{4.10} into \eqref{4.8} and letting $\varepsilon\rightarrow 0^+$, we obtain
\begin{align*}
&\frac{\lambda}{2(n-1)}\int_\Omega f^{t+r+1}\eta^2+\left(\frac{[1-(2-p)_+]t}{2}-\frac{4(n-1)\mu^2}{\lambda}-\frac{2(q-1)^2}{n-1}\right)\int_\Omega f^{t+p-2}|\nabla f|^2\eta^2\\
&+\left(\frac{[1-(2-q)_+]t}{2}-\frac{4(n-1)\mu^2}{\lambda}-\frac{2(q-1)^2}{n-1}\right)\int_\Omega f^{t+q-2}|\nabla f|^2\eta^2\\
\le&4(n-1)\kappa\int_{\Omega}\left(f^{t+p}+f^{t+q}\right)\eta^2+\frac{(4+3|p-2|+|q-2|)^2}{2[1-(2-p)_+]t}\int_\Omega f^{t+p}|\nabla \eta|^2\\
&+\frac{(4+3|q-2|+|p-2|)^2}{2[1-(2-q)_+]t}\int_\Omega f^{t+q}|\nabla \eta|^2.
\end{align*}
Now, we choose $\overline{t}_0=\overline{t}_0(n,p,q,\mu)$ large enough, such that
\begin{align*}
\frac{[1-(2-p)_+]t}{4}-\frac{4(n-1)\mu^2}{\lambda}-\frac{2(q-1)^2}{n-1}\ge0
\end{align*}
and
\begin{align*}
\frac{[1-(2-q)_+]t}{4}-\frac{4(n-1)\mu^2}{\lambda}-\frac{2(q-1)^2}{n-1}\ge0
\end{align*}
hold true for any $t\ge\overline{t}_0$. Hence, for any $t\ge\overline{t}_0$, the following inequality holds true
\begin{align*}
&\frac{\lambda}{2(n-1)}\int_\Omega f^{t+r+1}\eta^2+\frac{[1-(2-p)_+]t}{4}\int_\Omega f^{t+p-2}|\nabla f|^2\eta^2+\frac{[1-(2-q)_+]t}{4}\int_\Omega f^{t+q-2}|\nabla f|^2\eta^2\\
\le&4(n-1)\kappa\int_{\Omega}\left(f^{t+p}+f^{t+q}\right)\eta^2+\frac{(4+3|p-2|+|q-2|)^2}{2[1-(2-p)_+]t}\int_\Omega f^{t+p}|\nabla \eta|^2\\
&+\frac{(4+3|q-2|+|p-2|)^2}{2[1-(2-q)_+]t}\int_\Omega f^{t+q}|\nabla \eta|^2.
\end{align*}
\medskip

\noindent\textbf{Case 2:} Suppose that
\begin{align*}
\frac{\partial h}{\partial x}(x,y)\le-\lambda y^{r-\frac{q}{2}}\quad(\lambda>0),
\end{align*}
then we can achieve that
\begin{align}\label{siu2}
-2\int_\Omega\frac{\partial h}{\partial x}(u,f)f\left(f^\frac{p}{2}+f^\frac{q}{2}\right)f_\varepsilon^t\eta^2\ge2\lambda\int_\Omega f^{r+1}f_\varepsilon^t\eta^2.
\end{align}

By Cauchy inequality, we have
\begin{align}\label{siu3}
\begin{split}
2\mu\int_\Omega|\nabla f|f^{\frac{r}{2}-\frac{1}{2}} \left(f^\frac{p}{2}+f^\frac{q}{2}\right)f_\varepsilon^t\eta^2\le\lambda\int_\Omega f^{r+1}f_\varepsilon^t\eta^2+\frac{2\mu^2}{\lambda}\int_\Omega (f^p+f^q)f^{-2}f^t_\varepsilon|\nabla f|^2\eta^2.
\end{split}
\end{align}
Substituting \eqref{siu2} and \eqref{siu3} into \eqref{4.8} and letting $\varepsilon\rightarrow 0^+$, we obtain
\begin{align*}
&\lambda\int_\Omega f^{t+r+1}\eta^2+\left(\frac{[1-(2-p)_+]t}{2}-\frac{2\mu^2}{\lambda}-\frac{2(q-1)^2}{n-1}\right)\int_\Omega f^{t+p-2}|\nabla f|^2\eta^2\\
&+\left(\frac{[1-(2-q)_+]t}{2}-\frac{2\mu^2}{\lambda}-\frac{2(q-1)^2}{n-1}\right)\int_\Omega f^{t+q-2}|\nabla f|^2\eta^2\\
\le&4(n-1)\kappa\int_{\Omega}\left(f^{t+p}+f^{t+q}\right)\eta^2+\frac{(4+3|p-2|+|q-2|)^2}{2[1-(2-p)_+]t}\int_\Omega f^{t+p}|\nabla \eta|^2\\
&+\frac{(4+3|q-2|+|p-2|)^2}{2[1-(2-q)_+]t}\int_\Omega f^{t+q}|\nabla \eta|^2.
\end{align*}
Now we pick $\overline{t}_0=\overline{t}_0(n,p,q,\mu)$ large enough, such that
\begin{align*}
    \frac{[1-(2-p)_+]t}{4}-\frac{2\mu^2}{\lambda}-\frac{2(q-1)^2}{n-1}\ge0
\end{align*}
and
\begin{align*}
\frac{[1-(2-q)_+]t}{4}-\frac{2\mu^2}{\lambda}-\frac{2(q-1)^2}{n-1}\ge0
\end{align*}
hold true for any $t\ge\overline{t}_0$. Hence, it is easy to see that, for any $t\ge\overline{t}_0$, the following inequality holds true
\begin{align*}
&\lambda\int_\Omega f^{t+r+1}\eta^2+\frac{[1-(2-p)_+]t}{4}\int_\Omega f^{t+p-2}|\nabla f|^2\eta^2+\frac{[1-(2-q)_+]t}{4}\int_\Omega f^{t+q-2}|\nabla f|^2\eta^2\\
\le&4(n-1)\kappa\int_{\Omega}\left(f^{t+p}+f^{t+q}\right)\eta^2+\frac{(4+3|p-2|+|q-2|)^2}{2[1-(2-p)_+]t}\int_\Omega f^{t+p}|\nabla \eta|^2\\
&+\frac{(4+3|q-2|+|p-2|)^2}{2[1-(2-q)_+]t}\int_\Omega f^{t+q}|\nabla \eta|^2.
\end{align*}

Combining above, we can achieve the following lemma
\begin{lemma}\label{lemma4.3}
Let $(M,~g)$ be a complete Riemannian manifold with $\mathrm{Ric}\ge-(n-1)\kappa$ ($\kappa\ge0$), and let $u$ be a solution of equation \eqref{1} in $\Omega\subset M$. Assume that $h\in C^1(\mathbb{R}\times\mathbb{R}^+)$ satisfies one of the following two conditions
\begin{itemize}
\item $\frac{\partial h}{\partial x}(x,y)\le0$,\quad $\left(\frac{\partial h}{\partial y}(x,y)\right)^2\le\mu^2 y^{r-2}$ \quad\mbox{and}\quad $h^2(x,y)\ge\lambda y^r$;

\item $\frac{\partial h}{\partial x}(x,y)\le-\lambda y^{r-\frac{q}{2}}$ \quad\mbox{and}\quad $\left(\frac{\partial h}{\partial y}(x,y)\right)^2\le\mu^2 y^{r-2}$,
\end{itemize}
where $r>q-1$, $\lambda>0$ and $\mu>0$. Then, there exists $\overline{t}_0=\overline{t}_0(n,p,q,\mu)>0$, such that for any $t\ge\overline{t}_0$, the following inequality
\begin{align*}
&\int_\Omega f^{t+r+1}\eta^2+t\int_\Omega f^{t+p-2}|\nabla f|^2\eta^2+t\int_\Omega f^{t+q-2}|\nabla f|^2\eta^2\\
\le&\mathcal{C}\kappa\int_{\Omega}\left(f^{t+p}+f^{t+q}\right)\eta^2+\frac{\mathcal{C}}{t}\int_\Omega( f^{t+p}+f^{t+q})|\nabla \eta|^2
\end{align*}
holds true, where $\mathcal{C}=\mathcal{C}(n,p,q,\lambda)>0$.
\end{lemma}

\subsection{$L^{\beta_1}$-bound of gradient for the solutions of \eqref{1} in a geodesic ball}\

By using Lemma \ref{lemma4.3}, we can achieve the following lemma:
\begin{lemma}\label{lemma4.4}
Let $(M,~g)$ be a complete Riemannian manifold with $\mathrm{Ric}\ge-(n-1)\kappa$ ($\kappa\ge0$, $n\ge3$), and let $u$ be a solution of equation \eqref{1} in $\Omega=B(o,R)\subset M$. Assume that $h\in C^1(\mathbb{R}\times\mathbb{R}^+)$ satisfies one of the following two conditions
\begin{itemize}
\item $\frac{\partial h}{\partial x}(x,y)\le0$,\quad $\left(\frac{\partial h}{\partial y}(x,y)\right)^2\le\mu^2 y^{r-2}$ \quad\mbox{and}\quad $h^2(x,y)\ge\lambda y^r$;

\item $\frac{\partial h}{\partial x}(x,y)\le-\lambda y^{r-\frac{q}{2}}$ \quad\mbox{and}\quad $\left(\frac{\partial h}{\partial y}(x,y)\right)^2\le\mu^2 y^{r-2}$,
\end{itemize}
where $r>q-1$, $\lambda>0$ and $\mu>0$. Let $R>\delta>0$ and
    \begin{align*}
        \mathcal{C}_2(1+\kappa R^2)\le t_0\le\mathcal{C}_3(1+\kappa R^2),
    \end{align*}
then the following estimate holds true
\begin{align*}
\|f\|_{L^{\beta_1}(\Omega_0)}\le\mathcal{C}_0\left[\left(\frac{1+\kappa R^2}{R^2}\right)^\frac{1}{r-p+1}+\left(\frac{1+\kappa R^2}{R^2}\right)^\frac{1}{r-q+1}\right]V^{\frac{1}{\beta_1}},
\end{align*}
where $\Omega_0=B(o,\frac{3}{4}R)$, $\beta_1=\frac{n(t_0+p)}{n-2}$, $\mathcal{C}_i=\mathcal{C}_i(n,p,q,\lambda,r)$ ($i=2,3$) and $\mathcal{C}_0=\mathcal{C}_0(n,p,q,\lambda,r,\delta)$.
\end{lemma}

\begin{proof}
By Lemma \ref{lemma4.3}, we have
\begin{align}\label{4.11}
\begin{split}
\int_\Omega f^{t+r+1}\eta^2+t\int_\Omega f^{t+p-2}|\nabla f|^2\eta^2\le\mathcal{C}\kappa\int_{\Omega}\left(f^{t+p}+f^{t+q}\right)\eta^2+\frac{\mathcal{C}}{t}\int_\Omega( f^{t+p}+f^{t+q})|\nabla \eta|^2,
\end{split}
\end{align}
where $t\ge\overline{t}_0(n,p,q,\mu)$ and $\mathcal{C}=\mathcal{C}(n,p,q,\lambda)>0$.

On the other hand, we have
\begin{align*}
\left|\nabla\right(f^{\frac{t+p}{2}}\eta\left)\right|^2&=\left|\frac{t+p}{2}f^{\frac{t+p}{2}-1}\eta\nabla f+f^{\frac{t+p}{2}}\nabla\eta\right|^2\\
&\le\frac{(t+p)^2}{2}f^{t+p-2}\left|\nabla f\right|^2\eta^2+2f^{t+p}\left|\nabla\eta\right|^2.
\end{align*}
Substituting the above inequality into \eqref{4.11}, we have
\begin{align*}
\int_\Omega f^{t+r+1}\eta^2+\frac{1}{t} \int_\Omega\left|\nabla\right(f^{\frac{t+p}{2}}\eta\left)\right|^2 \le\mathcal{C}\kappa\int_{\Omega}\left(f^{t+p}+f^{t+q}\right)\eta^2+\frac{\mathcal{C}}{t}\int_\Omega( f^{t+p}+f^{t+q})|\nabla \eta|^2,
\end{align*}
where $t\ge\overline{t}_0(n,p,q,\mu)$ and $\mathcal{C}=\mathcal{C}(n,p,q,\lambda)>0$.

Now, by using Saloff-Coste's Sobolev inequality, we arrive at
\begin{align}\label{4.12}
\begin{split}
&\int_\Omega f^{t+r+1}\eta^2+\frac{1}{t}\exp\left\{-C_n(1+\sqrt{\kappa}R)\right\} V^\frac{2}{n}R^{-2}\left\|f^{\frac{t+p}{2}} \eta\right\|^2_{L^\frac{2n}{n-2}(\Omega)}\\
\le&\left(\mathcal{C}\kappa +\frac{1}{tR^2}\right)\int_{\Omega}f^{t+p}\eta^2+\frac{\mathcal{C}}{t}\int_\Omega f^{t+p}|\nabla \eta|^2+\mathcal{C}\kappa \int_{\Omega}f^{t+q}\eta^2+\frac{\mathcal{C}}{t}\int_\Omega f^{t+q}|\nabla \eta|^2.
\end{split}
\end{align}
Hence, by choosing $t_0$ large enough, we have
\begin{align}\label{4.13}
\begin{split}
&t_0\int_\Omega f^{t_0+r+1}\eta^2+\exp\left\{-C_n(1+\sqrt{\kappa}R)\right\} V^\frac{2}{n}R^{-2}\left\|f^{\frac{t_0+p}{2}}\eta\right\|^2_{L^\frac{2n}{n-2}(\Omega)}\\
\le&\left(\mathcal{C}\kappa t_0+\frac{1}{R^2}\right)\int_{\Omega}f^{t_0+p}\eta^2+\mathcal{C}\int_\Omega f^{t_0+p}|\nabla \eta|^2+\mathcal{C}\kappa t_0\int_{\Omega}f^{t_0+q}\eta^2+\mathcal{C}\int_\Omega f^{t_0+q}|\nabla \eta|^2.
\end{split}
\end{align}
Furthermore, we let $0\le\eta\le1$. Denote
\begin{align*}
\widehat{\Omega}_1=\left\{x\in\Omega:~f\ge\left(4\mathcal{C}\kappa+\frac{4}{t_0R^2}\right)^\frac{1}{r-p+1}\right\},
\end{align*}
we have
\begin{align}\label{4.14}
\left(\mathcal{C}\kappa t_0+\frac{1}{R^2}\right)\int_{\widehat{\Omega}_1}f^{t_0+p}\eta^2\le\frac{t_0}{4}\int_\Omega f^{t_0+r+1}\eta^2
\end{align}
and
\begin{align}\label{4.15}
\left(\mathcal{C}\kappa t_0+\frac{1}{R^2}\right)\int_{\Omega\setminus\widehat{\Omega}_1}f^{t_0+p}\eta^2 \le\left(\frac{4}{t_0}\right)^{\frac{t_0+p}{r-p+1}}\left(\mathcal{C}\kappa t_0+\frac{1}{R^2}\right)^{\frac{t_0+r+1}{r-p+1}}V.
\end{align}
Denote
\begin{align*}
\widehat{\Omega}_2=\left\{x\in\Omega:~f\ge\left(4\mathcal{C}\kappa\right)^\frac{1}{r-q+1}\right\},
\end{align*}
we have
\begin{align}\label{4.16}
    \mathcal{C}\kappa t_0\int_{\widehat{\Omega}_2}f^{t_0+q}\eta^2\le\frac{t_0}{4}\int_\Omega f^{t_0+r+1}\eta^2
\end{align}
and
\begin{align}\label{4.17}
    \mathcal{C}\kappa t_0\int_{\Omega\setminus\widehat{\Omega}_2}f^{t_0+q}\eta^2\le t_04^{\frac{t_0+q}{r-q+1}}\left(\mathcal{C}\kappa\right)^\frac{t_0+r+1}{r-q+1}V.
\end{align}

We denote $\Omega_0=B(o,\frac{3R}{4})$ and choose $\eta_1\in C_0^\infty(\Omega)$ satisfying
\begin{align*}
\begin{cases}
0\le\eta_1\le1,\quad\eta_1 \equiv1\quad in~\Omega_0;\\[3mm]
|\nabla\eta_1|\le\frac{C}{R}.
\end{cases}
\end{align*}
Let
$$\eta=\eta_1^\frac{t_0+r+1}{r-q+1}.$$
Then, by using Young inequality and the structure of $\eta$, we have
\begin{align}\label{4.18}
\mathcal{C}\int_\Omega f^{t_0+p}|\nabla \eta|^2=&\mathcal{C}\frac{(t_0+r+1)^2}{(r-q+1)^2}\int_\Omega f^{t_0+p} \eta_1^\frac{2(t_0+q)}{r-q+1}|\nabla\eta_1|^2\nonumber\\
\le&\frac{t_0}{4}\int_\Omega f^{t_0+r+1}\eta^2+\frac{r-p+1}{t_0+r+1}\left[\frac{\mathcal{C}(t_0+r+1)^2}{(r-q+1)^2}\right]^\frac{t_0+r+1}{r-p+1}\left[\frac{4(t_0+p)}{t_0(t_0+r+1)}\right]^\frac{t_0+p}{r-p+1}\nonumber\\
&\cdot\int_\Omega \eta_1^{\frac{2(q-p)(t_0+r+1)}{(r-q+1)(r-p+1)}}|\nabla\eta_1|^\frac{2(t_0+r+1)}{r-p+1}\\
\le&\frac{t_0}{4}\int_\Omega f^{t_0+r+1}\eta^2+\frac{\mathcal{C}_1}{t_0}\left(\mathcal{C}\mathcal{C}_1t_0^2\right)^\frac{t_0+r+1}{r-p+1}\left(\frac{\mathcal{C}_1}{t_0}\right)^\frac{t_0+p}{r-p+1}\int_\Omega |\nabla\eta_1|^\frac{2(t_0+r+1)}{r-p+1}\nonumber\\
\le&\frac{t_0}{4}\int_\Omega f^{t_0+r+1}\eta^2+\mathcal{C}^\frac{t_0+r+1}{r-p+1} \mathcal{C}_1^\frac{2t_0+2r+2}{r-p+1}t_0^\frac{t_0+r+1}{r-p+1}\left(\frac{C}{R}\right)^\frac{2(t_0+r+1)}{r-p+1}V,\nonumber
\end{align}
where $\mathcal{C}_1=\mathcal{C}_1(p,q,r)>0$. Similar to the proof of \eqref{4.18}, we have
\begin{align}\label{4.19}
\mathcal{C}\int_\Omega f^{t_0+q}|\nabla \eta|^2\le\frac{t_0}{4}\int_\Omega f^{t_0+r+1}\eta^2+\mathcal{C}^\frac{t_0+r+1}{r-q+1}\mathcal{C}_1^\frac{2t_0+2r+2}{r-q+1}t_0^\frac{t_0+r+1}{r-q+1}\left(\frac{C}{R}\right)^\frac{2(t_0+r+1)}{r-q+1}V.
\end{align}
Substituting $\eta=\eta_1^\frac{t_0+r+1}{r-q+1}$, \eqref{4.14}, \eqref{4.15}, \eqref{4.16}, \eqref{4.17}, \eqref{4.18} and \eqref{4.19} into \eqref{4.15} leads to
\begin{align*}
    &\exp\left\{-C_n(1+\sqrt{\kappa}R)\right\}V^\frac{2}{n}R^{-2}\left\|f^{\frac{t_0+p}{2}}\right\|^2_{L^\frac{2n}{n-2}(\Omega_0)}\\
    \le&\left(\frac{4}{t_0}\right)^{\frac{t_0+p}{r-p+1}}\left(\mathcal{C}\kappa t_0+\frac{1}{R^2}\right)^{\frac{t_0+r+1}{r-p+1}}V+t_04^{\frac{t_0+q}{r-q+1}}\left(\mathcal{C}\kappa\right)^\frac{t_0+r+1}{r-q+1}V\\
    &+\mathcal{C}^\frac{t_0+r+1}{r-p+1}\mathcal{C}_1^\frac{2t_0+2r+2}{r-p+1}t_0^\frac{t_0+r+1}{r-p+1}\left(\frac{C}{R}\right)^\frac{2(t_0+r+1)}{r-p+1}V\\
    &+\mathcal{C}^\frac{t_0+r+1}{r-q+1}\mathcal{C}_1^\frac{2t_0+2r+2}{r-q+1}t_0^\frac{t_0+r+1}{r-q+1}\left(\frac{C}{R}\right)^\frac{2(t_0+r+1)}{r-q+1}V.
\end{align*}
Hence, we arrive at
\begin{align*}
    &\exp\left\{-C_n(1+\sqrt{\kappa}R)\right\}V^\frac{2}{n}R^{-2}\left\|f^{\frac{t_0+p}{2}}\right\|^2_{L^\frac{2n}{n-2}(\Omega_0)}\\
    \le&\mathcal{C}^{\frac{2t_0+r+p+1}{r-p+1}}t_0^{-\frac{t_0+p}{r-p+1}}\left(\kappa t_0+\frac{1}{R^2}\right)^{\frac{t_0+r+1}{r-p+1}}V+\mathcal{C}^{\frac{2t_0+r+q+1}{r-q+1}}t_0\kappa^\frac{t_0+r+1}{r-q+1}V\\
    &+\mathcal{C}^\frac{5(t_0+r+1)}{r-p+1}\left(\frac{t_0}{R^2}\right)^\frac{t_0+r+1}{r-p+1}V+\mathcal{C}^\frac{5(t_0+r+1)}{r-q+1}\left(\frac{t_0}{R^2}\right)^\frac{t_0+r+1}{r-q+1}V.
\end{align*}
By using the above inequality, we obtain
\begin{align*}
    \left(\int_{\Omega_0}f^{(t_0+p)\frac{n}{n-2}}\right)^{\frac{n-2}{n}}
    \le&\exp\left\{C_n(1+\sqrt{\kappa}R)\right\}V^\frac{n-2}{n}\bigg\{\mathcal{C}^{\frac{2t_0+r+p+1}{r-p+1}}t_0^{-\frac{t_0+p}{r-p+1}}\left(\kappa t_0+\frac{1}{R^2}\right)^{\frac{t_0+r+1}{r-p+1}}R^2\\
    &+\mathcal{C}^{\frac{2t_0+r+q+1}{r-q+1}}t_0\kappa^\frac{t_0+r+1}{r-q+1}R^2+\mathcal{C}^\frac{5(t_0+r+1)}{r-p+1}t_0^{\frac{t_0+r+1}{r-p+1}}
    \left(\frac{1}{R^2}\right)^\frac{t_0+p}{r-p+1}\\
    &+\mathcal{C}^\frac{5(t_0+r+1)}{r-q+1}t_0^{\frac{t_0+r+1}{r-q+1}}\left(\frac{1}{R^2}\right)^\frac{t_0+q}{r-q+1}\bigg\}.
\end{align*}
Taking power of $\frac{1}{t_0+p}$ of the both sides of the above inequality respectively, we obtain
\begin{align*}
\|f\|_{L^{\beta_1}(\Omega_0)}\le&\mathcal{C}\exp\left\{\frac{C_n(1+\sqrt{\kappa}R)}{t_0+p}\right\}V^{\frac{1}{\beta_1}}\bigg\{t_0^{-\frac{t_0+p}{r-p+1}}\left(\kappa t_0+\frac{1}{R^2}\right)^{\frac{t_0+r+1}{r-p+1}}R^2\\
&+t_0\kappa^\frac{t_0+r+1}{r-q+1}R^2+t_0^{\frac{t_0+r+1}{r-p+1}}\left(\frac{1}{R^2}\right)^\frac{t_0+p}{r-p+1}+t_0^{\frac{t_0+r+1}{r-q+1}}\left(\frac{1}{R^2}\right)^\frac{t_0+q}{r-q+1}\bigg\}^\frac{1}{t_0+p},
\end{align*}
where $\mathcal{C}=\mathcal{C}(n,p,q,\lambda,r)$ and $\beta_1=\frac{n(t_0+p)}{n-2}$.

By using the inequality $$(a_1+a_2+a_3+a_4)^b\le4^b(a_1^b+a_2^b+a_3^b+a_4^b)$$
$(a_i\ge0,~b>0)$, we have
\begin{align}\label{4.20}
\|f\|_{L^{\beta_1}(\Omega_0)}\le&\mathcal{C}\exp\left\{\frac{C_n(1+\sqrt{\kappa}R)}{t_0+p}\right\}V^{\frac{1}{\beta_1}}4^\frac{1}{t_0+p}\bigg\{t_0^{-\frac{1}{r-p+1}}\left(\kappa t_0R^2+1\right)^{\frac{t_0+r+1}{(r-p+1)(t_0+p)}}R^\frac{-2}{r-p+1}\nonumber\\
&+t_0^{\frac{1}{t_0+p}}\kappa^\frac{t_0+r+1}{(r-q+1)(t_0+p)}R^\frac{2}{t_0+p}+t_0^{\frac{t_0+r+1}{(r-p+1)(t_0+p)}}\left(\frac{1}{R^2}\right)^\frac{1}{r-p+1}\\
&+t_0^{\frac{t_0+r+1}{(r-q+1)(t_0+p)}}\left(\frac{1}{R^2}\right)^\frac{t_0+q}{(r-q+1)(t_0+p)}\bigg\}\nonumber\\
:=&\mathcal{C}\exp\left\{\frac{C_n(1+\sqrt{\kappa}R)}{t_0+p}\right\}V^{\frac{1}{\beta_1}}4^\frac{1}{t_0+p}(I_1+I_2+I_3+I_4).\nonumber
\end{align}

Let
\begin{align}\label{4.21}
\mathcal{C}_2(1+\kappa R^2)\le t_0\le\mathcal{C}_3(1+\kappa R^2),
\end{align}
where $\mathcal{C}_i=\mathcal{C}_i(n,p,q,\lambda,r)$ ($i=2,3$); and note that
\begin{align*}
\lim_{\omega\to+\infty}{\omega^\frac{1}{\omega}}=1,
\end{align*}
which will be used in the estimate of $I_i$.

For $I_1$, we have
\begin{align}\label{4.22}
   \begin{split}
        I_1=&t_0^{-\frac{1}{r-p+1}}\left(\kappa t_0R^2+1\right)^{\frac{t_0+r+1}{(r-p+1)(t_0+p)}}R^\frac{-2}{r-p+1}\\
    =&\left(\kappa+\frac{1}{t_0R^2}\right)^\frac{1}{r-p+1}(1+\kappa t_0R^2)^\frac{1}{t_0+p}\\
    \le&\mathcal{C}\left(\kappa+\frac{1}{R^2}\right)^\frac{1}{r-p+1}[t_0(1+\kappa R^2)]^\frac{1}{t_0+p}\\
    \le&\mathcal{C}\left(\frac{1+\kappa R^2}{R^2}\right)^\frac{1}{r-p+1}.
   \end{split}
\end{align}

For $I_2$, since $R>\delta>0$, we have
\begin{align}\label{4.23}
\begin{split}
I_2=&t_0^{\frac{1}{t_0+p}}\kappa^\frac{t_0+r+1}{(r-q+1)(t_0+p)}R^\frac{2}{t_0+p}\\
=&t_0^{\frac{1}{t_0+p}}\kappa^{\frac{1}{r-q+1}}R^{-\frac{2(q-p)}{(t_0+p)(r-q+1)}}\left(\kappa R^2\right)^{\frac{r-p+1}{(r-q+1)(t_0+p)}}\\
\le&t_0^{\frac{1}{t_0+p}}\kappa^{\frac{1}{r-q+1}}R^{-\frac{2(q-p)}{(t_0+p)(r-q+1)}}\left(1+\kappa R^2\right)^{\frac{r-p+1}{(r-q+1)(t_0+p)}}\\
\le&\mathcal{C}_0\kappa^{\frac{1}{r-q+1}}\\
\le&\mathcal{C}_0\left(\frac{1+\kappa R^2}{R^2}\right)^\frac{1}{r-q+1},
\end{split}
\end{align}
where $\mathcal{C}_0=\mathcal{C}_0(n,p,q,\lambda,r,\delta)$.

For $I_3$, we have
\begin{align}\label{4.24}
I_3=t_0^{\frac{t_0+r+1}{(r-p+1)(t_0+p)}}\left(\frac{1}{R^2}\right)^\frac{1}{r-p+1}=\left(\frac{t_0}{R^2}\right)^\frac{1}{r-p+1}t_0^\frac{1}{t_0+p}
\le\mathcal{C}\left(\frac{1+\kappa R^2}{R^2}\right)^\frac{1}{r-p+1}.
\end{align}

For $I_4$, since $R>\delta>0$, we have
\begin{align}\label{4.25}
\begin{split}
I_4=&t_0^{\frac{t_0+r+1}{(r-q+1)(t_0+p)}}\left(\frac{1}{R^2}\right)^\frac{t_0+q}{(r-q+1)(t_0+p)}\\
=&\left(\frac{t_0}{R^2}\right)^\frac{1}{r-q+1}\left(\frac{1}{R^2}\right)^{\frac{q-p}{(r-q+1)(t_0+p)}}t_0^{\frac{r-p+1}{(r-q+1)(t_0+p)}}\\
\le&\mathcal{C}_0\left(\frac{1+\kappa R^2}{R^2}\right)^\frac{1}{r-q+1}.
\end{split}
\end{align}

By using \eqref{4.21}, we have
\begin{align}\label{4.26}
\exp\left\{\frac{C_n(1+\sqrt{\kappa}R)}{t_0+p}\right\}4^\frac{1}{t_0+p}\le\mathcal{C}.
\end{align}
Substituting \eqref{4.22}, \eqref{4.23}, \eqref{4.24}, \eqref{4.25} and \eqref{4.26} into \eqref{4.20}, we arrive at
\begin{align*}
\|f\|_{L^{\beta_1}(\Omega_0)}\le\mathcal{C}_0\left[\left(\frac{1+\kappa R^2}{R^2}\right)^\frac{1}{r-p+1}+\left(\frac{1+\kappa R^2}{R^2}\right)^\frac{1}{r-q+1}\right]V^{\frac{1}{\beta_1}}.
\end{align*}
Therefore, we complete the proof of Lemma \ref{lemma4.4}.

\end{proof}

\subsection{Moser iteration for solutions of \eqref{1}}\

By using the integral inequality \eqref{4.12}, we can achieve the following lemma:
\begin{lemma}\label{lemma4.5}
Let $(M,~g)$ be a complete Riemannian manifold with $\mathrm{Ric}\ge-(n-1)\kappa$ ($\kappa\ge0$, $n\ge3$), and let $u$ be a solution of equation \eqref{1} in $\Omega=B(o,R)\subset M$. Assume that $h\in C^1(\mathbb{R}\times\mathbb{R}^+)$ satisfies one of the following two conditions
\begin{itemize}
\item $\frac{\partial h}{\partial x}(x,y)\le0$,\quad $\left(\frac{\partial h}{\partial y}(x,y)\right)^2\le\mu^2 y^{r-2}$\quad\mbox{and}\quad $h^2(x,y)\ge\lambda y^r$;

\item $\frac{\partial h}{\partial x}(x,y)\le-\lambda y^{r-\frac{q}{2}}$ \quad\mbox{and}\quad $\left(\frac{\partial h}{\partial y}(x,y)\right)^2\le\mu^2 y^{r-2}$,
    \end{itemize}
    where $r>q-1$, $\lambda>0$ and $\mu>0$. Set
    \begin{align*}
        \xi=\exp\left\{\frac{(n-2)C_n(1+\sqrt{\kappa}R)}{2(t_0+p)}\right\}V^{\frac{-1}{\beta_1}}\left(\kappa t_0R^2 +1\right)^\frac{n-2}{2(t_0+p)}\|f\|_{L^{\beta_{1}}(B(o,\frac{3R}{4}))},
    \end{align*}
    then
    \begin{align*}
        \|f\|_{L^{\infty}(B(o,sR))}\le\mathcal{C}\left(\xi+\xi^\frac{2(t_0+p)}{2(t_0+p)-(n-2)(q-p)}\right)\frac{1}{(\tau-s)^\frac{2(n-2)}{2(t_0+p)-(n-2)(q-p)}},
    \end{align*}
    where $0<s<\tau<\frac{1}{2}$ and $t_0$ is large enough.
\end{lemma}
\begin{proof}
By using \eqref{4.12}, we have
    \begin{align*}
            &\exp\left\{-C_n(1+\sqrt{\kappa}R)\right\}V^\frac{2}{n}R^{-2}\left\|f^{\frac{t+p}{2}}\eta\right\|^2_{L^\frac{2n}{n-2}(\Omega)}\\
        \le&\mathcal{C}\left(\kappa t +\frac{1}{R^2}\right)\int_{\Omega}f^{t+p}\eta^2+\mathcal{C}\int_\Omega f^{t+p}|\nabla \eta|^2+\mathcal{C}\kappa t \int_{\Omega}f^{t+q}\eta^2+\mathcal{C}\int_\Omega f^{t+q}|\nabla \eta|^2.
    \end{align*}

Now, we denote $$r_m=sR+\frac{\tau-s}{\alpha^{m-1}}R$$ and $\Omega_m=B(o,r_m)$, where $\alpha=\sqrt{\frac{n}{n-2}}$ and $0<s<\tau<\frac{1}{2}$; and then choose $\eta_m\in C_0^\infty(\Omega_m)$ satisfying
\begin{align*}
\begin{cases}
0\le\eta_m\le1,\quad\eta_m\equiv1\quad in~\Omega_{m+1};\\
|\nabla\eta_m|\le\frac{C\alpha^m}{(\tau-s)R}.
\end{cases}
\end{align*}
Substituting $\eta$ by $\eta_m$ in the above integral inequality, we can easily verify that
\begin{align*}
&\exp\left\{-C_n(1+\sqrt{\kappa}R)\right\}V^\frac{2}{n}R^{-2}\left\|f^{\frac{t+p}{2}}\right\|^2_{L^\frac{2n}{n-2}(\Omega_{m+1})}\\
\le&\mathcal{C}\left(\kappa t +\frac{1}{R^2}\right)\int_{\Omega_m}f^{t+p}+\frac{\mathcal{C}\alpha^{2m}}{(\tau-s)^2R^2}\int_{\Omega_m} f^{t+p}+\mathcal{C}\kappa t \int_{\Omega_m}f^{t+q}+\frac{\mathcal{C}\alpha^{2m}}{(\tau-s)^2R^2}\mathcal{C}\int_{\Omega_m} f^{t+q}\\
\le&\mathcal{C}\left(\kappa t +\frac{1}{R^2}+\frac{\alpha^{2m}}{(\tau-s)^2R^2}\right)\int_{\Omega_m}(f^{t+p}+f^{t+q})\\
\le&\mathcal{C}\left(\kappa t +\frac{1}{R^2}+\frac{\alpha^{2m}}{(\tau-s)^2R^2}\right)\left(1+\|f\|^{q-p}_{L^\infty(B(o,\tau R))}\right)\int_{\Omega_m}f^{t+p}.
\end{align*}

Next, we choose
$$\beta_1=\frac{n(t_0+p)}{n-2}\quad \mbox{and}\quad \beta_{m+1}=\frac{n\beta_m}{n-2},$$
and let $t=t_m$ such that $t_m+p=\beta_m$. Then it follows from the above integral inequality that
\begin{align*}
&\exp\left\{-C_n(1+\sqrt{\kappa}R)\right\}V^\frac{2}{n}\left(\int_{\Omega_{m+1}}f^{\beta_{m+1}}\right)^{\frac{n-2}{n}}\\
\le&\mathcal{C}\left(\kappa t_mR^2+1+\frac{\alpha^{2m}}{(\tau-s)^2}\right)\left(1+\|f\|^{q-p}_{L^\infty(B(o,\tau R))}\right)\int_{\Omega_m}f^{\beta_m}.
\end{align*}
Taking power of $\frac{1}{\beta_m}$ on the both sides of the above inequality, we obtain
\begin{align*}
    &\exp\left\{-\frac{C_n(1+\sqrt{\kappa}R)}{\beta_m}\right\}V^\frac{2}{n\beta_m}\|f\|_{L^{\beta_{m+1}}(\Omega_{m+1})}\\
    \le&\mathcal{C}^\frac{1}{\beta_m}\left(\kappa t_mR^2 +1+\frac{\alpha^{2m}}{(\tau-s)^2}\right)^\frac{1}{\beta_m}\left(1+\|f\|^{q-p}_{L^\infty(B(o,\tau R))}\right)^\frac{1}{\beta_m}\|f\|_{L^{\beta_{m}}(\Omega_{m})}.
\end{align*}
Hence, we arrive at
\begin{align*}
\|f\|_{L^{\beta_{m+1}}(\Omega_{m+1})}\le&\exp\left\{\frac{C_n(1+\sqrt{\kappa}R)}{\beta_m}\right\}V^\frac{-2}{n\beta_m}\mathcal{C}^\frac{1}{\beta_m}
\left(\kappa t_mR^2 +1+\frac{\alpha^{2m}}{(\tau-s)^2}\right)^\frac{1}{\beta_m}\\
&\cdot\left(1+\|f\|^{q-p}_{L^\infty(B(o,\tau R))}\right)^\frac{1}{\beta_m}\|f\|_{L^{\beta_{m}}(\Omega_{m})}\\
\le&\exp\left\{\frac{C_n(1+\sqrt{\kappa}R)}{\beta_m}\right\}V^\frac{-2}{n\beta_m}\mathcal{C}^\frac{1}{\beta_m}\left(\kappa t_0R^2 +\frac{1}{(\tau-s)^2}\right)^\frac{1}{\beta_m}\\
&\cdot\left(\frac{n}{n-2}\right)^\frac{m}{\beta_m}\left(1+\|f\|^{q-p}_{L^\infty(B(o,\tau R))}\right)^\frac{1}{\beta_m}\|f\|_{L^{\beta_{m}}(\Omega_{m})}.
\end{align*}
Noting
\begin{align*}
0<s<\tau<\frac{1}{2},\quad\sum_{m=1}^\infty{\frac{1}{\beta_m}}=\frac{n}{2\beta_1}\quad \mbox{and}\quad\sum_{m=1}^\infty{\frac{m}{\beta_m}}=\frac{n^2}{4\beta_1},
\end{align*}
we have
\begin{align*}
\|f\|_{L^{\infty}(B(o,sR))}\le&\mathcal{C}^{\frac{n-2}{2(t_0+p)}}\exp\left\{\frac{(n-2)C_n(1+\sqrt{\kappa}R)}{2(t_0+p)}\right\}V^{\frac{-1}{\beta_1}}\left(\frac{n}{n-2}\right)^\frac{n^2}{4\beta_1}\\
&\cdot\left(\kappa t_0R^2 +\frac{1}{(\tau-s)^2}\right)^\frac{n}{2\beta_1}\left(1+\|f\|^{q-p}_{L^\infty(B(o,\tau R))}\right)^\frac{n}{2\beta_1}\|f\|_{L^{\beta_{1}}(B(o,\tau R))}.
\end{align*}
Hence, we arrive at
\begin{align*}
\|f\|_{L^{\infty}(B(o,sR))}\le&\mathcal{C}\exp\left\{\frac{(n-2)C_n(1+\sqrt{\kappa}R)}{2(t_0+p)}\right\}V^{\frac{-1}{\beta_1}}\left(\kappa t_0R^2 +1\right)^\frac{n-2}{2(t_0+p)}\\
&\cdot\left(1+\|f\|^\frac{(n-2)(q-p)}{2(t_0+p)}_{L^\infty(B(o,\tau R))}\right)\|f\|_{L^{\beta_{1}}(B(o,\tau R))}\frac{1}{(\tau-s)^\frac{n-2}{t_0+p}}.
\end{align*}

Set
\begin{align*}
\xi=\exp\left\{\frac{(n-2)C_n(1+\sqrt{\kappa}R)}{2(t_0+p)}\right\}V^{\frac{-1}{\beta_1}}\left(\kappa t_0R^2 +1\right)^\frac{n-2}{2(t_0+p)}\|f\|_{L^{\beta_{1}}(B(o,\frac{3R}{4}))},
\end{align*}
we have
\begin{align*}
\|f\|_{L^{\infty}(B(o,sR))}\le\mathcal{C}\xi\frac{1}{(\tau-s)^\frac{n-2}{t_0+p}}+\mathcal{C}\xi
\frac{\|f\|^\frac{(n-2)(q-p)}{2(t_0+p)}_{L^\infty(B(o,\tau R))}}{(\tau-s)^\frac{n-2}{t_0+p}}.
\end{align*}
By applying Young's inequality for the second integral on the right-hand side of the above inequality under $p\neq q$, we readily obtain
\begin{align*}
\|f\|_{L^{\infty}(B(o,sR))}\le&\frac{1}{2}\|f\|_{L^{\infty}(B(o,\tau R))}+\mathcal{C}\xi\frac{1}{(\tau-s)^\frac{n-2}{t_0+p}}+\left[\frac{(n-2)(q-p)}{t_0+p}\right]^{\frac{(n-2)(q-p)}{2(t_0+p)-(n-2)(q-p)}}\\
&\cdot\frac{2(t_0+p)-(n-2)(q-p)}{2(t_0+p)}(\mathcal{C}\xi)^{\frac{2(t_0+p)}{2(t_0+p)-(n-2)(q-p)}}\frac{1}{(\tau-s)^\frac{2(n-2)}{2(t_0+p)-(n-2)(q-p)}}\\
\le&\frac{1}{2}\|f\|_{L^{\infty}(B(o,\tau R))}+\mathcal{C} \left(\xi+\xi^\frac{2(t_0+p)}{2(t_0+p)-(n-2)(q-p)}\right)\frac{1}{(\tau-s)^\frac{2(n-2)}{2(t_0+p)-(n-2)(q-p)}}.
\end{align*}
Furthermore, it is worth to mention that the above inequality is also true under $p=q$. Setting $$\psi_0(y)=\left\|f\right\|_{L^\infty\left(B\left(o,yR\right)\right)},$$
we obtain
\begin{align*}
\psi_0(s)\le\frac{1}{2}\psi_0(\tau)+\mathcal{C}\left(\xi+\xi^\frac{2(t_0+p)}{2(t_0+p)-(n-2)(q-p)}\right)
\frac{1}{(\tau-s)^\frac{2(n-2)}{2(t_0+p)-(n-2)(q-p)}}.
\end{align*}
By using Lemma \ref{lemma2.2} and noticing that
\begin{align*}
\lim_{t_0\to+\infty}\frac{2(n-2)}{2(t_0+p)-(n-2)(q-p)}=0,
\end{align*}
we finish the proof of Lemma \ref{lemma4.5}.

\end{proof}

\subsection{Proof of Theorem \ref{theorem1.2}}
\begin{proof}
By using Lemma \ref{lemma4.5} and choosing $s=\frac{1}{4}$ and $\tau=\frac{3}{8}$, we have
\begin{align}\label{4.27}
\|f\|_{L^{\infty}(B(o,\frac{R}{4}))}\le\mathcal{C}\left(\xi+\xi^\frac{2(t_0+p)}{2(t_0+p)-(n-2)(q-p)}\right),
\end{align}
where
\begin{align*}
\xi=\exp\left\{\frac{(n-2)C_n(1+\sqrt{\kappa}R)}{2(t_0+p)}\right\}V^{\frac{-1}{\beta_1}}\left(\kappa t_0R^2 +1\right)^\frac{n-2}{2(t_0+p)}\|f\|_{L^{\beta_{1}}(B(o,\frac{3R}{4}))}
\end{align*}
and $t_0$ is large enough. Let
\begin{align*}
\mathcal{C}_2(1+\kappa R^2)\le t_0\le\mathcal{C}_3(1+\kappa R^2),
\end{align*}
we can derive that
\begin{align*}
\xi\le\mathcal{C}\|f\|_{L^{\beta_{1}}(B(o,\frac{3R}{4}))}V^{\frac{-1}{\beta_1}}.
\end{align*}
By using Lemma \ref{lemma4.4}, we arrive at
\begin{align}\label{4.28}
\xi\le\mathcal{C}_0\left[\left(\frac{1+\kappa R^2}{R^2}\right)^\frac{1}{r-p+1}+\left(\frac{1+\kappa R^2}{R^2}\right)^\frac{1}{r-q+1}\right],
\end{align}
where $\mathcal{C}_0=\mathcal{C}_0(n,p,q,\lambda,r,\delta)$ and $R\ge\delta$.

For any $\theta>1$, we can choose $\mathcal{C}_2$ large enough, such that
\begin{align*}
\frac{2(t_0+p)}{2(t_0+p)-(n-2)(q-p)}<\theta.
\end{align*}
Noting
$$1\le\frac{2(t_0+p)}{2(t_0+p)-(n-2)(q-p)}<\theta$$
and combining \eqref{4.27} and \eqref{4.28} together, we can achieve that
\begin{align*}
\|f\|_{L^{\infty}(B(o,\frac{R}{4}))}\le\widetilde{\mathcal{C}}\left[\left(\frac{1+\kappa R^2}{R^2}\right)^\frac{1}{r-p+1}+\left(\frac{1+\kappa R^2}{R^2}\right)^\frac{\theta}{r-q+1}\right]
\end{align*}
where $\widetilde{\mathcal{C}}=\widetilde{\mathcal{C}}(n,p,q,\lambda,r,\delta,\theta)$. Hence, we complete the proof of Theorem \ref{theorem1.1}.
\end{proof}

\section{\textbf{Some further applications}}

Next, we give several examples by using the above theorems.
\medskip

\noindent{\bf Example 1.} If we let $h(u,|\nabla u|^2)=au^s|\nabla u|^r$ in \eqref{1}, then \eqref{1} is of the following form:
\begin{align}\label{5.1}
\Delta_pu+\Delta_qu+au^s|\nabla u|^r=0,
\end{align}
where $s$ and $r$ are two real numbers. For this equation, we can take the same arguments as in Section 4 to conclude the following:

\begin{theorem} Let $M$ ($\dim(M)\ge3$) be a complete non-compact Riemannian manifold with non-negative Ricci curvature. Assume that $r>q-1\geq p-1>0$, $as\le0$ and $a\neq0$ in \eqref{5.1}. If $u$ is a positive solution to equation \eqref{5.1} on $M$, which satisfies
\begin{enumerate}
\item $u$ is bounded on $M$,

\item $u$ has a positive lower bound on $M$,
\end{enumerate}
then $u$ is a trivial constant solution.
\end{theorem}

\begin{proof}
Using the assumptions in $u$, we can see that there exists $l_1,l_2>0$, such that
\begin{align*}
l_1\le u\le l_2.
\end{align*}
Since $h(u,|\nabla u|^2)=au^s|\nabla u|^r$, we can see that
\begin{align*}
\dfrac{\partial h}{\partial x}(x,y)\bigg|_{(u,|\nabla u|^2)}=asu^{s-1}|\nabla u|^r\quad and \quad \dfrac{\partial h}{\partial y}(x,y)\bigg|_{(u,|\nabla u|^2)}=\frac{ar}{2}u^s|\nabla u|^{r-2}.
\end{align*}
Choosing $\mu=\frac{|ar|}{2}\max\{l_1^s,~l_2^s\}$ and $\lambda=a^2\min\{l_1^{2s},~l_2^{2s}\}$, and using $as\le0$, we obtain the following
\begin{align*}
&\dfrac{\partial h}{\partial x}(x,y)\bigg|_{(u,|\nabla u|^2)}\le0,\\
&\left(\frac{\partial h}{\partial y}(x,y)\right)^2\bigg|_{(u,|\nabla u|^2)}=\frac{a^2r^2}{4}x^{2s}y^{r-2}\bigg|_{(u,|\nabla u|^2)}\le\mu^2 y^{r-2}\bigg|_{|\nabla u|^2},\\
&h^2(x,y)\bigg|_{(u,|\nabla u|^2)}=a^2x^{2s}y^{r}\bigg|_{(u,|\nabla u|^2)}\ge\lambda y^r\bigg|_{|\nabla u|^2}.
\end{align*}
Therefore, using Corollary \ref{corollary1.3} and (3) of Remark 1, we finish the proof.
\end{proof}
\medskip

\noindent{\bf Example 2.} In fact, one may also consider the following equation:
\begin{align}\label{5.2}
\Delta_pu+\Delta_qu-u=0,
\end{align}
which can be obtained by letting $h(u,|\nabla u|^2)=-u$ in \eqref{1}.

\begin{theorem}
Let $1<p\le q<2$ and $M$ ($\dim(M)\ge3$) is a complete non-compact Riemannian manifold with non-negative Ricci curvature. If $u$ is a solution to equation \eqref{5.2} on $M$, then $u\equiv0$.
\end{theorem}

\begin{proof}
Since $h(u,|\nabla u|^2)=-u$, we can see that
\begin{align*}
\dfrac{\partial h}{\partial x}(x,y)\bigg|_{(u,|\nabla u|^2)}=-1\quad and \quad \dfrac{\partial h}{\partial y}(x,y)\bigg|_{(u,|\nabla u|^2)}=0.
\end{align*}
Choosing $r=\frac{q}{2}$, $\lambda=1$ and $\mu=1$, we obtain the following
\begin{align*}
\frac{\partial h}{\partial x}(x,y)\bigg|_{(u,|\nabla u|^2)}=-1\le-1\times y^0\bigg|_{|\nabla u|^2}=-\lambda y^{r-\frac{q}{2}}\bigg|_{|\nabla u|^2},
\end{align*}
\begin{align*}
\left(\frac{\partial h}{\partial y}(x,y)\right)^2\bigg|_{(u,|\nabla u|^2)}=0\le1\times y^{\frac{q}{2}-2}\bigg|_{|\nabla u|^2}=\mu^2 y^{r-2}\bigg|_{|\nabla u|^2}.
\end{align*}
Therefore, using Corollary \ref{corollary1.3} and (3) of Remark 1, we finish the proof.
\end{proof}
\medskip

\noindent {\it\bf{Acknowledgements}}: The authors are supported by National Natural Science Foundation of China (Grant No. 12431003).
\medskip

\noindent {\it\bf{Conflict of interest statement}}: The authors declare that there are no conflict of interests.

\medskip
\noindent {\it\bf{Data availability statement}}: No data was used in the research described in this paper.
\medskip
\medskip

\label{appendixB}
\bibliographystyle{plain}

\end{document}